\documentclass{amsart}
\usepackage{graphicx}
\usepackage{amssymb}
\newtheorem{thm}{Theorem}[section]
\newtheorem{cor}[thm]{Corollary}
\newtheorem{lem}[thm]{Lemma}

\theoremstyle{definition}
\newtheorem{defn}[thm]{Definition}
\theoremstyle{remark}
\newtheorem{rem}[thm]{Remark}
\numberwithin{equation}{section}

\newcommand{\abs}[1]{\left\vert#1\right\vert}
\newcommand{\set}[1]{\left\{#1\right\}}

\newcommand{\eps}{\varepsilon}

\newcommand{\dbar}{\bar\partial}

\newcommand{\bd}{\partial}
\newcommand{\til}{\tilde}

\newcommand{\R}{\mathbb{R}}
\newcommand{\p}{\pi}
\newcommand{\grad}{\nabla}

\DeclareMathOperator{\re}{Re}
\DeclareMathOperator{\im}{Im}

\begin{document}

\title[Hartogs Domains and the Diederich-Forn{\ae}ss Index]{Hartogs Domains and the Diederich-Forn{\ae}ss Index}%
\author{Muhenned Abdulsahib}%
\address{University of Dhi-Qar}%
\email{mabdulsa@email.uark.edu}%

\author{Phillip S. Harrington}%
\address{SCEN 309, 1 University of Arkansas, Fayetteville, AR 72701}%
\email{psharrin@uark.edu}%

\subjclass[2010]{32U10, 32T27}

\begin{abstract}
We study a geometric property of the boundary on Hartogs domains which can be used to find upper and lower bounds for the Diederich-Forn{\ae}ss Index. Using this, we are able to show that under some reasonable hypotheses on the set of weakly pseudoconvex points, the Diederich-Forn{\ae}ss Index for a Hartogs domain is equal to one if and only if the domain admits a family of good vector fields in the sense of Boas and Straube. We also study the analogous problem for a Stein neighborhood basis, and show that under the same hypotheses if the Diederich-Forn{\ae}ss Index for a Hartogs domain is equal to one then the domain admits a Stein neighborhood basis.
\end{abstract}
\maketitle

\tableofcontents


\section{Introduction }

Let $\Omega\subset\mathbb{C}^n$ be a bounded pseudoconvex domain with $C^2$ boundary.  Recall that $\Omega$ is pseudoconvex if and only if $-\log\delta$ is a plurisubharmonic function on $\Omega$, where $\delta(z)$ denotes the distance from $z$ to the boundary of $\Omega$.  Note that the signed distance function $\tilde\delta$ (i.e., $\tilde\delta=\delta$ on $\Omega^c$ and $\tilde\delta=-\delta$ on $\Omega$) is $C^2$ near the boundary of $\Omega$ by a result of Krantz and Parks \cite{KrPa81}, so the level curves of $-\log\delta$ near $b\Omega$ form an exhaustion of $\Omega$ by $C^2$ pseudoconvex domains.  We say that $-\log\delta$ is an unbounded plurisubharmonic exhaustion function for $\Omega$.

In \cite{DiFo77b}, Diederich and Forn{\ae}ss proved that every bounded pseudoconvex domain with $C^2$ boundary also admits a bounded plurisubharmonic exhaustion function.  In particular, there exists an exponent $0<\tau<1$ and a $C^2$ defining function $\rho$ for $\Omega$ such that $-(-\rho)^\tau$ is strictly plurisubharmonic on $\Omega$.  Recall that $\rho$ is a defining function for $\Omega$ if $\Omega=\{z:\rho(z)<0\}$ and $d\rho\neq 0$ on $b\Omega$.  Since we will be working on smooth domains, we call $\tau$ a Diederich-Forn{\ae}ss exponent for $\Omega$ if there exists a smooth defining function $\rho$ for $\Omega$ such that $-(-\rho)^\tau$ is plurisubharmonic on $\Omega$, and let the Diederich-Forn{\ae}ss Index of $\Omega$ denote the supremum over all Diederich-Forn{\ae}ss exponents.  In \cite{DiFo77a}, Diederich and Forn{\ae}ss show that for any $0<\tau<1$, there exists a bounded pseudoconvex domain $\Omega$ with smooth boundary such that $\tau$ is not a Diederich-Forn{\ae}ss exponent for $\Omega$.  These domains have come to be known as worm domains, and we will examine them more closely in a later section (see Definition \ref{worm} below).

Much recent interest in the Diederich-Forn{\ae}ss Index has stemmed from its connections to regularity for the Bergman Projection.  The Bergman Projection $P$ for $\Omega$ is the orthogonal projection from $L^2(\Omega)$ onto the closed subspace of $L^2$ holomorphic functions.  We say that a smooth domain $\Omega$ satisfies Condition $R$ if $P$ preserves the space $C^\infty(\overline\Omega)$.  Extending a result of Fefferman \cite{Fef74}, Bell and Ligocka \cite{BeLi80} proved that if $f:\Omega_1\rightarrow\Omega_2$ is a biholomorphic map and $\Omega_1$ and $\Omega_2$ are smooth, bounded, pseudoconvex domains satisfying Condition $R$, then $f$ extends smoothly to the boundary of $\Omega_1$.  Given that the Riemann Mapping Theorem fails in several complex variables, this result gives us a critical tool to evaluate biholomorphic equivalence classes of domains by showing that in some cases a biholomorphism must preserve the geometry of the boundary (in contrast to the one variable case).  For example, a biholomorphism which extends to the boundary must preserve the signature of the Levi form for the boundary, while there is no corresponding curvature which must be preserved in one variable.  Unfortunately, Condition $R$ is known to fail on the worm domains of Diederich and Forn{\ae}ss by work of Barrett \cite{Bar92} and Christ \cite{Chr96}.

Under an additional technical hypothesis, it is known that Condition $R$ is satisfied when the Diederich-Forn{\ae}ss Index is equal to one by work of Kohn \cite{Koh99}, Harrington \cite{Har11}, and Pinton and Zampieri \cite{PiZa14}.  Hence, there is great interest in identifying cases in which the Diederich-Forn{\ae}ss Index is equal to one.  For example, Forn{\ae}ss and Herbig have shown that it suffices for $\Omega$ to admit a defining function which is plurisubharmonic on $b\Omega$ \cite{FoHe07,FoHe08}.  Other recent results have been obtained in \cite{KLP16}, \cite{Liu17b} and \cite{Har18}.

Our goal in this paper is to better understand some open questions regarding the Diederich-Forn{\ae}ss Index by considering the special case of Hartogs domains.  A domain $\Omega\subset\mathbb{C}^2$ is said to be Hartogs if
\begin{equation}
\label{eq:Hartogs}
  (e^{i\theta}z,w)\in\Omega\text{ whenever }(z,w)\in\Omega\text{ and }\theta\in\mathbb{R}.
\end{equation}
The symmetry of such domains makes them relatively tractable for analysis, so they are frequently studied to gain a deeper understanding of phenomena that are difficult to approach on more general classes of domains (for example, see \cite{BoSt92}, \cite{FuSt02}, or \cite{ChFu05}).  The worm domain of Diederich and Forn{\ae}ss is a Hartogs domain, and we will see that we can use our results to recover recent results on the worm domain due to Liu \cite{Liu17a} and Yum \cite{Yum18}.

We will see in this paper that, as with the worm domain, the presence of an annulus in the boundary of a Hartogs domain plays a critical role in the value of the Diederich-Forn{\ae}ss Index.  Our primary innovation is to show that there is a curvature term on such an annulus which can be used to find upper and lower bounds for the Diederich-Forn{\ae}ss Index.  We use $\tilde\delta$ to denote the signed distance function, i.e., $\tilde\delta(z)=-\delta(z)$ on $\overline\Omega$ and $\tilde\delta(z)=\delta(z)$ outside of $\overline\Omega$.  The result of Krantz and Parks \cite{KrPa81} shows that $\tilde\delta$ is $C^2$ in a neighborhood of $b\Omega$.  One can check that second derivatives of $\tilde\delta$ can be derived from the second fundamental form of $b\Omega$, and hence measure the extrinsic curvature of $b\Omega$.  We will see that in the coordinates given by \eqref{eq:Hartogs}, $\frac{\partial^2\tilde\delta}{\partial w\partial\bar z}$ is the critical term to study on an annulus in the boundary.  Since this term is obtained from the Hessian of the signed distance function, it depends on the second fundamental form of the boundary, so it is a geometric invariant.  In particular, if $2\re\left(e^{i\theta}\frac{\partial}{\partial w}\right)$ is the normal vector at a point $p$ in the annulus for some $\theta\in\mathbb{R}$, then $2\re\left(e^{i\theta}\frac{\partial}{\partial w}\right)\frac{\partial\tilde\delta}{\partial\bar z}(p)=0$ (this will follow from \eqref{eq:gradd} below).  Hence
\begin{align*}
  \frac{\partial^2\tilde\delta}{\partial w\partial\bar z}(p)&=i e^{-i\theta}\left(\frac{1}{2}\im\left(e^{i\theta}\frac{\partial^2\tilde\delta}{\partial w\partial x}(p)\right)+\frac{i}{2}\im\left(e^{i\theta}\frac{\partial^2\tilde\delta}{\partial w\partial y}(p)\right)\right)\\
  &=i e^{-i\theta}\left(\frac{1}{2}\mathrm{II}_p\left(\im\left(e^{i\theta}\frac{\partial}{\partial w}\right),\frac{\partial}{\partial x}\right)+\frac{i}{2}\mathrm{II}_p\left(\im\left(e^{i\theta}\frac{\partial}{\partial w}\right),\frac{\partial}{\partial y}\right)\right),
\end{align*}
where $z=x+iy$ and $\mathrm{II}_p$ denotes the second fundamental form at $p$.  In particular,
\[
  \abs{\frac{\partial^2\tilde\delta}{\partial w\partial\bar z}(p)}^2=\frac{1}{4}\abs{\mathrm{II}_p\left(\im\left(e^{i\theta}\frac{\partial}{\partial w}\right),\frac{\partial}{\partial x}\right)}^2+\frac{1}{4}\abs{\mathrm{II}_p\left(\im\left(e^{i\theta}\frac{\partial}{\partial w}\right),\frac{\partial}{\partial y}\right)}^2,
\]
so in some sense $\abs{\frac{\partial^2\tilde\delta}{\partial w\partial\bar z}(p)}$ measures the pairing of the maximal complex sub-bundle of the tangent space (i.e., the span of $\frac{\partial}{\partial x}$ and $\frac{\partial}{\partial y}$ at $p$) with the totally real tangent vector that is orthogonal to this space (i.e., $2\im\left(e^{i\theta}\frac{\partial}{\partial w}\right)$) with respect to the second fundamental form.  We will see (see \eqref{alfi} below) that $\frac{\partial^2\tilde\delta}{\partial w\partial\bar z}$ also corresponds to the coefficients of D'Angelo's one-form $\alpha$, which is know to be a crucial in the study of Condition $R$ (see Section 5.9 in \cite{Str10} for more information).

Our first result is an upper bound for the Diederich-Forn{\ae}ss Index:
\begin{thm}\label{main}
Let $\Omega\subset\mathbb{C}^2$ be a Hartogs domain with smooth boundary, and suppose that for some $w\in\mathbb{C}$, $B>A>0$ and $C>0$ the annulus $M =\{(z,w):\ A\leq\abs{z}^2 \leq B\}$ is in $b\Omega$ and $\abs{\frac{\bd^2\til{\delta}}{\bd w\bd \bar{z}}}>\frac{C}{\abs{ z}}$ on $M$.
Then the Diederich-Forn{\ae}ss Index is at most $\frac{ \pi}{2 C \log \frac{ B}{A}+\pi}$.
  \end{thm}	

To establish a lower bound for the Diederich-Forn{\ae}ss Index, we will find it helpful to place a mild regularity condition on the set of weakly pseudoconvex points.
\begin{defn}
\label{defn:regular_points}
  Let $\Omega\subset\mathbb{C}^2$ be a Hartogs domain with coordinates given by \eqref{eq:Hartogs}.  Let $W\subset b\Omega$ denote the set of weakly pseudoconvex points, and let
  \begin{align*}
    M_1&=\left\{p\in W:\frac{\partial\tilde\delta}{\partial z}(p)=0\right\},\\
    M_2&=\left\{p\in W:\frac{\partial\tilde\delta}{\partial z}(p)\neq 0\right\}.
  \end{align*}
  Note that we could replace $\tilde\delta$ with any $C^1$ defining function $\rho$ without altering the definitions of $M_1$ and $M_2$.  We say that $\Omega$ has regular weakly pseudoconvex points if $M_1\cap\bar M_2=\emptyset$, $M_1$ has finitely many connected components, and each connected component $K$ of $M_1$ has either of the following two properties:
  \begin{enumerate}
    \item If $(z,w)\in K$, then $(rz,w)\in K$ for all $0\leq r\leq 1$.
    \item If $(z,w)\in K$, then $(0,w)\notin K$.
  \end{enumerate}
  In case $(1)$, we say that $K$ is disk-like, and in case $(2)$, we say that $K$ is annulus-like.
\end{defn}
  We note that condition (1) can be weakened considerably, since there are more general regions on which Lemma 3.1 holds, but we have adopted the present definition for the sake of clarity.

  To illustrate this condition, we suppose that $\Omega$ is a complete Reinhardt domain, i.e., if $(z,w)\in\Omega$ then $(r_1 e^{i\theta_1}z,r_2 e^{i\theta_2}w)\in\Omega$ for all $0\leq r_1,r_2\leq 1$ and $\theta_1,\theta_2\in\mathbb{R}$.  A complete Reinhardt domain is known to be pseudoconvex if and only if $\log\tau(\Omega)=\{(\log|z|,\log|w|):(z,w)\in\Omega\}$ is a convex subset of $\mathbb{R}^2$ (see Theorem 3.28 in \cite{Ran86}, for example).  Using this characterization, one easily checks that on a pseudoconvex complete Reinhardt domain, $M_1$ has at most one connected component and this component must be disk-like, while $M_2$ consists of the remaining points $(z,w)\in b\Omega$ such that $b\log\tau(\Omega)$ is weakly convex at $(\log|z|,\log|w|)$ (except for possible points in $M_2$ at which $w=0$).  Hence, $M_1\cap\bar M_2=\emptyset$ whenever $M_2$ contains only finitely many connected components (since we always assume that $\Omega$ has smooth boundary). While $M_1$ may have a more complicated structure on general Hartogs domains, it would be interesting to know if it remains true that $\Omega$ has regular weakly pseudoconvex points whenever $M_1$ and $M_2$ both have only finitely many connected components, as is true in the case of complete Reinhardt domains.

With this additional hypothesis, we have the following:
  \begin{thm}\label{tctokh}
  	Let $\Omega\subset\mathbb{C}^2$ be a Hartogs domain with smooth boundary and regular weakly pseudoconvex points, and suppose that every   annulus-like connected component $K$ of $\left\{p\in b\Omega:\frac{\partial\tilde\delta}{\partial z}(p)=0\text{ and }\frac{\partial^2\tilde\delta}{\partial z\partial\bar z}(p)=0\right\}$ admits constants $B\geq A>0$ and $C>0$ such that $A\leq\abs{z}^2 \leq B$ on $K$ and $\abs{\frac{\bd^2\til{\delta}}{\bd w\bd \bar{z}}}\leq \frac{C}{\abs{ z}}$ on $K$. Then the Diederich-Forn{\ae}ss Index is at least $\frac{ \pi}{2 C \log \frac{ B}{A}+\pi}$.
  \end{thm}

If $\Omega$ is a worm domain (see Definition \ref{worm}), this allows us to recover a result of Liu \cite{Liu17a}.
\begin{cor}
\label{cor:worm_domain}
  Let $ \Omega_r$ be a worm domain with weakly pseudoconvex points given by the annulus $M_r=\{(z,w):1\leq\abs{z}\leq r, w=0\}$ for some $r>1$. Then the Diederich-Forn{\ae}ss Index for $\Omega_r$ is equal to $\frac{ \pi}{ \log  r^2+\pi}.$
\end{cor}

\begin{rem}
\label{rem:worm_domains}
  Here and elsewhere in this paper we use Diederich and Forn{\ae}ss's original definition of the worm domain.  Many recent papers, including \cite{Liu17a}, choose a parametrization such that the annulus in the boundary is given by
  \[
   M_\beta=\set{(z,w):-\beta+\frac{\pi}{2}\leq\log|z|^2\leq\beta-\frac{\pi}{2},w=0}
  \]
  for $\beta>\frac{\pi}{2}$.  These definitions are equivalent after a re-scaling under the relationship $r=\exp(\beta-\frac{\pi}{2})$.  With this parametrization, the Diederich-Forn{\ae}ss Index for the worm domain is equal to $\frac{\pi}{2\beta}$, which is the value computed in \cite{Liu17a}.
\end{rem}

As an application of these theorems, we use the Diederich-Forn{\ae}ss Index to study a sufficient condition for Condition $R$.  In \cite{BoSt91}, Boas and Straube introduced a sufficient condition for Condition $R$ which we will refer to as the good vector field condition.  We will use the refined version of this condition presented in \cite{BoSt99} (the further refinement presented in \cite{Str10} is not needed on the domains which we are considering).  We will define this condition precisely later, but for now we note that we can use our methods to prove the following equivalence:
\begin{thm}\label{thm:good_vector_field}
  Let $\Omega$ be a Hartogs domain with smooth boundary and regular weakly pseudoconvex points such that any annulus-like connected component of the weakly pseudoconvex points is an annulus.  Then the Diederich-Forn{\ae}ss Index is equal to one if and only if there exists a family of good vector fields on $b\Omega$.
\end{thm}	

As in several recent papers (e.g., \cite{Har18} and \cite{Yum18}), we will also consider the natural analogue of the Diederich-Forn{\ae}ss Index on the complement of $\Omega$.  Here, the case is complicated by the fact that not every pseudoconvex domain $\Omega$ admits a Stein neighborhood basis.  Recall that $\Omega$ admits a Stein neighborhood basis if for every open set $U$ containing $\overline\Omega$ there exists a pseudoconvex domain $\Omega_U$ such that $\overline\Omega\subset\Omega_U\subset U$.  For sufficiently large winding number $r$, Diederich and Forn{\ae}ss have already shown in \cite{DiFo77a} that the worm domain is a smooth, bounded, pseudoconvex domain that does not admit a Stein neighborhood basis.  Fortunately, the existence of a Stein neighborhood basis on the domains that we are studying is completely characterized by work of Bedford and Forn{\ae}ss \cite{BeFo78}.  In particular, we have
\begin{thm}\label{ostnb}
  Let $\Omega\subset\mathbb{C}^2$ be a Hartogs domain with smooth boundary, and suppose that for some $w\in\mathbb{C}$, $B>A>0$ and $C>0$ the set of weakly pseudoconvex points is equal to the annulus $M =\{(z,w):\ A\leq\abs{z}^2\leq B\}$ in $b\Omega$ and $\abs{\frac{ \bd^2\tilde{\delta}}{\bd z\bd\bar{w}}}<\frac{ \pi}{2\sqrt{ A}\abs{ log\frac{ A}{B}}}$  when $\abs{ z}=\sqrt{A}$.  Then a Stein neighborhood basis for $\bar{ \Omega} $ exists. If $\abs{\frac{\bd^2\til{\delta}}{\bd z\bd\bar{w}}}>\frac{\pi}{2\sqrt{A}\abs{\log\frac{A}{B}}}$ when $\abs{z}=\sqrt{A} $ then no Stein neighborhood basis exists.
\end{thm}

As an interesting consequence, Theorems \ref{main} and \ref{ostnb} allow us to show the following:
\begin{cor}\label{cor:st_nbhd_existence}
  Let $\Omega$ be a Hartogs domain with smooth boundary such that the set of weakly pseudoconvex points is equal to an annulus, and suppose that the Diederich-Forn{\ae}ss Index for $\Omega$ is equal to one.  Then $\Omega$ admits a Stein neighborhood basis.
\end{cor}
Given the connection between Condition $R$ and the Diederich-Forn{\ae}ss Index studied in \cite{Koh99}, \cite{Har11}, and \cite{PiZa14}, this result can be seen as a parallel result to Zeytuncu's Theorem 8 (and Remark 6) in \cite{Zey14}.  Zeytuncu shows that on the Hartogs domain $\Omega_g=\{(z,w)\in\mathbb{C}^2:|w|<1\text{ and }|z|<|g(w)|\}$, where $g$ is a bounded holomorphic function on the unit disk, Condition $R$ implies that $\overline\Omega_g$ admits a Stein neighborhood basis.  Note that $\Omega_g$ is necessarily not smooth, so Zeytuncu's result applies to a different class of Hartogs domains from those considered in Corollary \ref{cor:st_nbhd_existence}.

Once we know that a Stein neighborhood basis exists, we have the analogue of Theorem \ref{main} on the complement of $\Omega$:
\begin{thm}\label{tntn1}
  Let $\Omega\subset\mathbb{C}^2$ be a Hartogs domain with smooth boundary, and suppose that for some $w\in\mathbb{C}$, $B>A>0$ and $C>0$ the annulus $M =\{(z,w):\ A\leq\abs{z}^2\leq B\}$ is in $b\Omega$ and $\abs{\frac{\bd^2\til{\delta}}{\bd w\bd \bar{z}}}>\frac{C}{\abs{ z}}$ on $M$.
  Assume there exists a smooth strictly positive function $h$
  such that
  	\[ \sigma=h\til{\delta}^{\tau}\]  is plurisubharmonic on $U\backslash\bar{\Omega}$, for some $\tau>1$ and some neighborhood $U$ of $\overline\Omega$. Then  $C<\frac{\pi}{2\log \frac{B}{A}} $ and  $ \tau>\frac{ \pi}{\pi-2C\log\frac{ B}{A}}.$
\end{thm}
Adopting terminology of \cite{Sah12}, we say that $\Omega$ admits a strong Stein neighborhood basis if the neighborhood basis is obtained by considering level curves of a plurisubharmonic function.  The hypotheses of Theorem \ref{tntn1} are even stronger, although they are known to be satisfied in many cases (see \cite{FoHe07}, \cite{FoHe08}, or \cite{Har18}).  There is not yet a standard terminology for this analogue of the Diederich-Forn{\ae}ss Index, although Yum \cite{Yum18} has suggested ``Steinness Index."

On the worm domain (see Definition \ref{worm}), Theorem \ref{tntn1} gives us:
\begin{cor}\label{simsim}
  Let $ \Omega_r$ be a worm domain with weakly pseudoconvex points given by the annulus $M_r=\{(z,w):1\leq\abs{z}\leq r, w=0\}$ for some $r>1$.  If $\abs{\log\frac{1}{r^2}}<\pi $, then a Stein neighborhood basis exists for $\Omega_r$ and if there exists $\tau>1$ and a smooth function $h>0$ such that $\rho=h\tilde{\delta}^{\tau}$ is plurisubharmonic on $U\backslash\overline\Omega$, then $\tau\geq \frac{\pi}{\pi-\log r^2}$.  If $\abs{\log\frac{1}{r^2}}>\pi$, then no Stein neighborhood basis exists.
\end{cor}
This is consistent with recent discoveries due to Yum \cite{Yum18}, after the re\-par\-a\-met\-riz\-at\-ion discussed in Remark \ref{rem:worm_domains}.

The authors would like to thank the anonymous referee for many helpful comments and corrections.  This paper is adapted from the PhD Thesis of the first author at the University of Arkansas under the direction of the second author.

   \section{Preliminaries}

\subsection{Notation and Definitions}  	

  Recall that a pseudoconvex domain with $C^2$ boundary is a domain $\Omega$ such that for any $C^2$ defining function $r$ and $p\in b\Omega$, the Levi form
  	\[
  	\langle Lr(z)t,t \rangle=
  	\sum\limits_{i,j=1}^n \frac{\partial^2 r}{\partial z_j \partial \overline{z}_j}(p)t_i \overline{t}_j\]
  	is nonnegative for all $t=(t_1, \cdots, t_n) \in \mathbb{C}^n$ with $\sum_{j=1}^{n} t_j(\frac{\partial r}{\partial z_j})(p)=0$.  If the Levi form is strictly positive for all $t\neq 0$ at some $p\in b\Omega$, we say that $\Omega$ is strictly pseudoconvex at $p$.

Since our primary goal in this paper is to study bounded plurisubharmonic exhaustion functions for a domain, we recall that a function $\varphi : \Omega \rightarrow \mathbb{R}$ is an exhaustion function for a domain $\Omega$  if the closure of $\{x\in \Omega |\varphi(x)<c\}$ is compact for all $c$ in the range of $\varphi$.  A $C^2$ function $\varphi$ is said to be plurisubharmonic on $\Omega$ if and only if
  	\[\sum\limits_{j,k=1}^n \frac{\partial^2 \varphi}{\partial z_j \partial \overline{z}_k}(z)t_j \overline{t}_k
  	\geq 0\]
  	for all $t=(t_1, \cdots, t_n) \in \mathbb{C}^n$ and $z\in\Omega$.



We adopt the original definition of the worm domain given by Diederich and Forn{\ae}ss in \cite{DiFo77a}:
\begin{defn}\label{worm}
	Let $\lambda:\mathbb{R}\rightarrow\mathbb{R} $ be a smooth function satisfying the following properties:
	\begin{enumerate}
	\item $\lambda(x)=0	$ if $x\leq 0 $
	\item $\lambda(x)>1 $ if $ x>1$
	\item $\lambda''(x)\geq100\lambda'(x) $ for all $x $
	\item $\lambda''(x)>0 $ if $x>0 $
	\item $\lambda'(x)>100 $ if $\lambda(x)>\frac{1}{2} .$
	\end{enumerate}
	For any $r>1$, define the function $\rho_r:\mathbb{C}\times\mathbb{C}\rightarrow\mathbb{R}$ as follows:
	$$\rho_r(z,w)=\abs{ w+e^{( i\log z\bar{z})}}^2-1+\lambda\left(\frac{1}{\abs{z}^2}-1\right)+\lambda\left(\abs{z}^2-r^2\right). $$
	 Then $\Omega_r=\{(z,w)\in\mathbb{C}\times\mathbb{C}: \rho_r(z,w)<0\}$ is called a worm domain.
	 \end{defn}

  As shown in \cite{DiFo77a}, the worm domain is a smooth pseudoconvex domain.  Also in \cite{DiFo77a}, it is shown that for any exponent $0<\tau<1$, there exists an $r>1$ such that the Diederich-Forn{\ae}ss Index for the worm domain $\Omega_r$ is less than $\tau$.  Part of our goal in this paper is to adapt Diederich and Forn{\ae}ss's argument to general Hartogs domains admitting an annulus in the boundary.

  Following \cite{BoSt99}, we say that $\Omega$ admits a family of good vector fields if for every $\eps>0$, there exists a vector field $X_{\eps}$ of type $(1,0)$ such that the coefficients of $X_\eps$ are smooth in a neighborhood $U_\eps$ of the set of the boundary points of $\Omega$ of infinite type, and the following conditions are satisfied:
  \begin{enumerate}
  \item $\abs{\arg X_{\eps}\rho}<\eps$	on $U_\eps$,  $C^{-1}<\abs{X_\eps\rho}<C$, and
  \item $\bd\rho[X_\eps,\frac{\bd}{\bd\bar{z}_j}]<\eps$ on $U_\eps$ for all $1\leq j\leq n$.
  \end{enumerate}
  Boas and Straube have shown that if a smooth, bounded domain $\Omega$ admits a family of good vector fields, then $\Omega$ satisfies Condition $R$ (see \cite{BoSt91} for the original version, \cite{BoSt99} for the version used here, and \cite{Str10} for a further refinement that is not needed in our setting).

   \subsection{Distance Function}
  For a smooth bounded domain $\Omega\subset\mathbb{R}^n$ with $C^2$ boundary, there exists a neighborhood $U$ of $b\Omega$ on which the signed distance function $\tilde\delta$ is also $C^2$ \cite{KrPa81}.  For $x\in U$, let $\pi(x)$ denoted the unique point in $b\Omega$ minimizing the distance to $x$ (the uniqueness of $\pi(x)$ on $U$ is also found in \cite{KrPa81}).  On this set $U$, we have
  \begin{equation}
    \label{eq:gradd}
    \nabla\tilde\delta(x)=\nabla\tilde\delta(\pi(x)).
  \end{equation}
  This follows from Theorem 4.8 in \cite{Fed59}.

  The following result can be found in  \cite{Wei75}, \cite{HeMc12}, and \cite{HaRa13}:
  \begin{thm}
  	For any smooth bounded domain $\Omega\subseteq \R^n $, there exists a neighborhood U of $b\Omega$ such that for all $x\in U$
  	\begin{equation}\label{eq:second} \nabla^2\til{\delta}(x)=\nabla^2\til{\delta}(\pi(x))\cdot \left(I+\til{\delta}(x)\nabla^2\til{\delta}(\pi(x))\right)^{-1}, \end{equation}
  where $I$ denotes the identity matrix and $\nabla^2\tilde\delta$ denotes the (real) Hessian of $\tilde\delta$.
  	
  \end{thm}

  With more calculations we can rewrite this result in complex coordinates.  For $1\leq j,k\leq n$, if we write $z_j=x_j+iy_j$, then we have
  \begin{equation} \label{eq:formula}
  \frac{ \bd^2 \til{\delta}}{\bd z_j\bd \bar{ z}_k}=\frac{ 1}{4}\frac{ \bd^2\til{\delta}}{\bd x_j\bd x_k}+\frac{ 1}{4}\frac{ \bd^2\til{\delta}}{\bd y_j\bd y_k}+\frac{ i}{4}\frac{ \bd^2\til{\delta}}{\bd x_j\bd y_k}-\frac{ i}{4}\frac{ \bd^2\til{\delta}}{\bd y_j\bd x_k}.\end{equation}

  When $\delta(z)$ is close to zero, the left-hand side of \eqref{eq:second} can be approximated by
  \begin{multline*}\nabla^2\til{\delta}(z)=
  \nabla^2\til{\delta}(\p(z))\left(I-\til{\delta}(z)\nabla^2\til{\delta}(\p(z))\right)+O((\delta(z))^2)\\
  =\nabla^2\til{\delta}(\p(z))-\til{\delta}(z)\left(\nabla^2\tilde{ \delta}(\p(z))\right)^2+O((\delta(z))^2),\end{multline*}
  where $\left(\nabla^2\tilde{ \delta}(\p(z))\right)^2=\left(\nabla^2\tilde{ \delta}(\p(z))\right)\left(\nabla^2\tilde{ \delta}(\p(z))\right)$ denotes matrix multiplication.

  Substituting \eqref{eq:formula}, we obtain
  \begin{multline*}
  \frac{ \bd^2\til{\delta}(z)}{\bd z_j\bd\bar{z}_k}= \frac{ \bd^2\til{\delta}(\p (z))}{\bd z_j\bd\bar{z}_k}\\
  -\til{\delta}(z)\sum_{\ell=1}^{n}
  \left(\frac{ \bd^2\til{\delta}(\p (z))}{\bd z_j\bd x_\ell} \cdot \frac{ \bd^2\til{\delta}(\p (z))}{\bd x_\ell \bd \bar{z}_k} +   \frac{ \bd^2\til{\delta}(\p (z))}{\bd z_j\bd y_\ell} \cdot \frac{ \bd^2\til{\delta}(\p (z))}{\bd y_\ell \bd \bar{z}_k} \right) +O((\delta(z))^2).
  \end{multline*}
  Using the identities $\frac{ \bd}{\bd x_j}=\frac{ \bd}{\bd z_j}+\frac{ \bd}{\bd\bar{z}_j}$ and $\frac{ \bd}{\bd y_j}=i\left(\frac{ \bd}{\bd z_j}-\frac{ \bd}{\bd\bar{z}_j}\right)$, we obtain
  \begin{multline}
  \label{dist}
  \frac{\bd^2\tilde{ \delta}(z)}{\bd z_j \bd \overline{ z}_k}=\frac{ \bd^2 \til{\delta}(\pi(z))}{\bd z_j \bd\overline{ z}_k}-\til{\delta}(z)\sum_{\ell=1}^{n}\Big(2\frac{ \bd^2 \til{\delta}(\pi(z))}{\bd z_j \bd\overline{ z}_\ell}\cdot\frac{ \bd^2 \til{\delta}(\pi(z))}{\bd z_\ell \bd\overline{ z}_k}
  \\
  +2\frac{\bd^2 \til{\delta}(\pi(z))}{\bd z_j \bd{ z}_\ell}\cdot\frac{ \bd^2 \til{\delta}(\pi(z))}{\bd \overline{z}_\ell \bd\overline{ z}_k}\Big)+O((\delta(z))^2).
\end{multline}

   \section{Proof of Main Result }
  \subsection{Upper bound for Diederich-Forn{\ae}ss Index}
  	 \begin{proof}[Proof of Theorem \ref{main}]
  	 	
Since $\Omega$ is  Hartogs, we may assume
 \[
 \til{ \delta}(z,w)=\rho(\abs{z}^2,w),\]
 for some function $\rho$ that is smooth for $(z,w)$ near the boundary of $\Omega$.  To emphasize that this is actually a function on $\mathbb{R}\times\mathbb{C}$, we introduce the notation $t=|z|^2$.  We will abbreviate $\rho_w=\frac{\partial\rho}{\partial w}$ and $\rho_t=\frac{\partial\rho}{\partial t}$.  To simplify notation, we will frequently identify the points $(|z|^2,w)$ and $(z,w)$.

For any $p\in\Omega$ satisfying $\pi(p)\in M$, the signed distance function  $\til{\delta}$ satisfies
\begin{equation}
\label{eq:z_derivative_vanishes}
\frac{ \bd\tilde{ \delta}}{\bd z}(\pi(p))=0,
\end{equation}
and
\begin{equation}
\label{eq:z_bar_z_derivative_vanishes}
\frac{ \bd^2\tilde{ \delta}}{\bd z\bd \overline{ z}}(\pi(p))=0.
\end{equation}
 From \eqref{eq:gradd},
 \begin{equation}
 \label{eq:z_derivative_vanishes_interior}
 \frac{ \bd\til{\delta}}{\bd z}(p)=0
 \end{equation}
 as well.  Hence,
 \begin{equation}\label{eq:deltaw}
   \|\rho_w(p)\|^2=\frac{1}{4}\|\nabla\tilde\delta(p)\|^2=\frac{1}{4}.
 \end{equation}
   If we take the derivative of both sides of \eqref{eq:deltaw} with respect to $t$, we get
   \begin{equation}\label{eq:phi}
   \rho_{t\bar{w}}(p)\rho_w(p)+\rho_{tw}(p)\rho_{\bar{w}}(p)=0.
   \end{equation}
 Since
 \[\frac{ \bd\til{\delta}}{\bd z}(z,w)=\rho_t(\abs{ z}^2,w)\overline{ z},\]
 and the first component of $\pi(p)$ is non-zero on the annulus $M$, \eqref{eq:z_derivative_vanishes} gives us
 \[\rho_t(\pi(p))=0.\]
 Since
 \[\frac{ \bd^2\til{\delta}}{\bd z\bd\overline{ z}}(z,w)=\rho_{tt}(\abs{z}^2,w)\abs{ z}^2+\rho_t(\abs{ z}^2,w),\]
 \eqref{eq:z_bar_z_derivative_vanishes} gives us
 \begin{equation}
 \label{eq:tt_derivative_vanishes}
 \rho_{tt}(\pi(p))=0.
 \end{equation}
  We also compute
  \begin{equation}
  \label{eq:wz_derivative}
  \frac{ \bd^2\tilde{\delta}}{\bd w \bd z}(z,w)=\rho_{tw}(\abs{ z}^2,w)\overline{ z},
  \end{equation}
  and
  \begin{equation}
  \label{eq:bar_w_z_derivative}
  \frac{ \bd^2\tilde{ \delta}}{\bd \overline{ w} \bd z}(z,w)=\rho_{t\bar{w}}(\abs{ z}^2,w)\overline{ z}.
  \end{equation}

 Furthermore, since
 \[\frac{ \bd^2\tilde{ \delta}}{\bd z \bd z}(z,w)=\rho_{tt}(\abs{ z}^2,w)\overline{ z}^2,\]
 then \eqref{eq:tt_derivative_vanishes} gives us
 \begin{equation}
 \label{eq:zz_derivative_vanishes}
 \frac{ \bd^2\tilde{ \delta}}{\bd z \bd z}(\pi(p))=0.
 \end{equation}
 Substituting \eqref{eq:z_bar_z_derivative_vanishes} and \eqref{eq:zz_derivative_vanishes} in \eqref{dist}, we obtain
 \begin{multline*}
   \frac{ \bd^2\til{\delta}}{\bd z\bd\overline{ z}}(p)=-2\til{\delta}(p)\Big(\frac{ \bd^2\tilde{ \delta}}{\bd z\bd \overline{ w}}(\pi(p))\cdot\frac{ \bd^2\til{\delta}}{\bd w\bd\overline{ z}}(\pi(p))\\+\frac{ \bd^2\til{\delta}}{\bd z\bd w}(\pi(p))\cdot\frac{ \bd^2 \til{\delta}}{\bd\overline{ w}\bd\overline{z}}(\pi(p))\Big)+O((-\til{\delta}(p))^2).
 \end{multline*}
 Now, \eqref{eq:bar_w_z_derivative} and \eqref{eq:wz_derivative} give us
 \begin{equation}\label{eq:tin}
 \frac{ \bd^2\til{\delta}}{\bd z\bd\overline{ z}}(p)=
 -4\til{\delta}(p)\abs{\rho_{tw}(\pi(p))}^{2} \abs{p_z}^{2}+O((-\til{\delta}(p))^{2}).
 \end{equation}

Assume that $-h(-\tilde\delta)^\tau$ is plurisubharmonic on $U\cap\Omega$ for some positive $C^2$ function $h$ and some neighborhood $U$ of $b\Omega$.  Using the Hartogs symmetry of $\Omega$, we can set
\[
  \hat h(z,w)=\int_0^{2\pi}h(e^{i\theta}z,w)d\theta
\]
and obtain another plurisubharmonic function $-\hat h(-\tilde\delta)^\tau$ which is rotationally symmetric in $z$.  We set $\sigma=-\hat h(-\tilde\delta)^\tau$.

  We denote the complex Hessian of $\sigma $ at the point $(z,w)$ in the direction $(\xi,\eta)$ by 	
  	\begin{multline*} H_\sigma\left(\abs{z}^{2},w,\xi,\eta\right)=\\
  \frac{\partial^{2}\sigma}{\partial z \bd \overline{ z} }(z,w)\abs{\xi}^{2}+2 \re\left( \frac{\partial^2 \sigma}{\partial\overline{z}\partial w}(z,w)\eta\overline{\xi}\right)+\frac{\partial^2\sigma}{\partial w\partial \bar{w}}(z,w)\abs{\eta}^2 .\end{multline*}
At $p=(p_z,p_w)$, we use \eqref{eq:z_derivative_vanishes_interior} to compute
\[
\frac{\partial^{2}\sigma}{\bd z \bd \overline{ z} }(p)=\hat{h}(p)\tau\left((-\til{\delta}(p))^{(\tau-1)}\frac{ \bd^2\til{\delta}}{\bd z\bd\bar{z}}(p)\right)-\newline\frac{ \bd^2\hat{ h}}{\bd z\bd \bar z}(p)(-\til{\delta}(p))^\tau.\]
Substituting \eqref{eq:tin}, we get
		\[
		\frac{ \bd^2\sigma}{\bd z\bd\bar{ z}}(p)=\left(4\hat{h}(p)\tau\abs{\rho_{tw}(\pi(p))}^{2} \abs{p_z}^{2}-\newline\frac{ \bd^2\hat{ h}}{\bd z\bd \bar z}(p)\right)(-\til{\delta}(p))^\tau+O((-\til{\delta}(p))^{\tau+1}).\]
Also, \eqref{eq:z_derivative_vanishes_interior} gives us
\begin{multline*}
 \frac{\partial^2 \sigma}{\partial z\partial\bar{w}}(p)=\\
 \hat{ h}(p)\tau(-\til{\delta}(p))^{\tau-1}\frac{ \bd^2\til{\delta}}{\bd z\bd\bar{ w}}(p)- (-\til{\delta}(p))^\tau\frac{\bd^2\hat{h}}{\bd z\bd \bar{ w}}(p)+\tau(-\til{\delta}(p))^{(\tau-1)}\frac{\bd \hat{h}}{\bd z }(p)\frac{ \bd\til{\delta}}{\bd\bar{w}}(p).
 \end{multline*}
Since $\rho$ is $C^3$, $|\rho_{t\bar w}(\pi(p))-\rho_{t\bar w}(p)|<O(-\tilde\delta(p))$, so \eqref{eq:bar_w_z_derivative} implies
\begin{multline*}
 \frac{\partial^2 \sigma}{\partial z\partial\bar{w}}(p)=\\
 \hat{ h}(p)\tau(-\til{\delta}(p))^{\tau-1}\rho_{t\bar w}(\pi(p))\bar p_z+\tau(-\til{\delta}(p))^{(\tau-1)}\frac{\bd \hat{h}}{\bd z }(p)\frac{ \bd\til{\delta}}{\bd\bar{w}}(p)+O((-\til{\delta}(p))^{\tau}).
 \end{multline*}
 Furthermore,
 \begin{multline*}
 \frac{ \bd^2\sigma}{\bd w\bd\bar{w}}(p)=\\
 \hat{ h}(p)\tau\left((-\til{\delta}(p))^{(\tau-1)}\frac{ \bd^2 \til{\delta}}{\bd w\bd\bar{w}}(p)+(1-\tau)(-\til{\delta}(p))^{(\tau-2)}\frac{\bd\til{\delta}}{\bd\bar{w}}(p)\frac{ \bd\til\delta}{\bd w}(p)\right)\\
 +
 \tau(-\til{\delta}(p))^{(\tau-1)}\frac{ \bd\hat{h}}{\bd \bar{w}}(p)\frac{ \bd\til{\delta}}{\bd w}(p)\\
 -(-\til{\delta}(p))^\tau\frac{ \bd^2\hat{h}}{\bd w\bd \bar{w}}(p)+\tau(-\til{\delta}(p))^{(\tau-1)}\frac{ \bd\hat{h}}{\bd w}(p)\frac{ \bd\til{\delta}}{\bd\bar w}(p),
 \end{multline*}
 so \eqref{eq:deltaw} gives us
 \[
 \frac{\partial^2\sigma}{\partial w\partial \bar{w}}(p)=\frac{1}{4}\tau(1-\tau)(-\til{\delta}(p))^{(\tau-2)}\hat{h}(p)+O((-\til{\delta}(p))^{(\tau-1)}).
 \]

%

  %


  %
   %
   %
   %
   %
   %
   %
   %
   %
   %

  Motivated by the different orders of vanishing in each term, we evaluate the complex Hessian in the direction $(\xi,\eta)=(\xi,(-\tilde\delta(p))\hat\eta)$, and obtain
  \begin{multline*}
  H_\sigma(\abs{p_z}^2,p_w,\xi,\eta)=\left(4\hat{h}(p)\tau\abs{\rho_{tw}(\pi(p))}^{2} \abs{p_z}^{2}-\newline\frac{ \bd^2\hat{ h}}{\bd z\bd \bar z}(p)\right)(-\til{\delta}(p))^\tau\abs{\xi}^{2}\\
  +2\re\left(\left(\hat{ h}(p)\tau(-\til{\delta}(p))^{\tau}\rho_{t\bar w}(\pi(p))\bar p_z+\tau(-\til{\delta}(p))^{\tau}\frac{\bd \hat{h}}{\bd z }(p)\frac{ \bd\til{\delta}}{\bd\bar{w}}(p)\right)\xi\bar{\hat\eta}\right)\\
  +\frac{1}{4}\tau(1-\tau)(-\til{\delta}(p))^{\tau}\hat{h}(p)\abs{\hat\eta}^2+O((-\tilde\delta(p))^{\tau+1}|(\xi,\eta)|).
  \end{multline*}
Dividing by $(-\til{\delta}(p))^\tau$ and letting $p\rightarrow\pi(p)$,
  we see that if the complex Hessian is positive definite on $U\cap\Omega$ for some neighborhood $U$ of $b\Omega$, then the form
  \begin{multline*}
  \lim_{p\rightarrow\pi(p)}\frac{H_\sigma(\abs{p_z}^2,p_w,\xi,\eta)}{(-\tilde\delta(p))^\tau}=\\
  \left(4\hat{h}(\pi(p))\tau\abs{\rho_{tw}(\pi(p))}^{2} \abs{\pi(p)_z}^{2}-\newline\frac{ \bd^2\hat{ h}}{\bd z\bd \bar z}(\pi(p))\right)\abs{\xi}^{2}\\
  +2\re\left(\left(\hat{ h}(\pi(p))\tau\rho_{t\bar w}(\pi(p))\overline{\pi(p)}_z+\tau\frac{\bd \hat{h}}{\bd z }(\pi(p))\frac{ \bd\til{\delta}}{\bd\bar{w}}(\pi(p))\right)\xi\bar{\hat\eta}\right)\\
  +\frac{1}{4}\tau(1-\tau)\hat{h}(\pi(p))\abs{\hat\eta}^2
  \end{multline*}
  must be positive semi-definite on $M$.  On $M$, we set
  \[
    a(p)=4\hat{h}(p)\tau\abs{\rho_{tw}(p)}^{2} \abs{p_z}^{2}-\newline\frac{ \bd^2\hat{h}}{\bd z\bd \bar z}(p),
  \]
  \[
    b(p)=\frac{1}{4}\tau(1-\tau)\hat{h}(p),
  \]
  and
  \[
    c(p)=\hat{ h}(p)\tau\rho_{t\bar w}(p)\bar p_z+\tau\frac{\bd \hat{h}}{\bd z }(p)\frac{ \bd\til{\delta}}{\bd\bar{w}}(p),
  \]
  then this is equivalent to requiring that the matrix $\begin{pmatrix}a(p)&c(p)\\\bar c(p)&b(p)\end{pmatrix}$ be positive semi-definite.  Since $0<\tau<1$ and $\hat h(p)$ is positive, $b(p)>0$ on $M$, so it suffices to show that the determinant $a(p)b(p)-|c(p)|^2$ is also non-negative.

   We compute
  	\[\frac{\bd}{\bd z}\hat{h}(\abs{z}^2,w)=\frac{\bd\hat{h}}{\bd t}(\abs{ z}^2,w)\cdot \bar{z},\]
and
  	\[\frac{ \bd^2}{\bd z\bd \bar{z}}\hat{h}(\abs{z}^2,w)=\frac{\bd \hat{h}}{\bd t}(\abs{z}^2,w)+\frac{\bd^2\hat{h}}{\bd t^2}(\abs{ z}^2,w)\abs{z}^2.\]
  	If we write $p_t=|p_z|^2$, we can simplify $\abs{c(p)}^2$ via  	
  	\begin{multline*}\abs{c(p)}^{2}=\abs{\tau\bar p_z\left(\hat{h}(p)\rho_{t\bar{w}}(p)+\frac{\partial\hat{h}}{\partial t}(p)\rho_{\bar{w}}(p)\right)}^2\\
  	=\tau^{2} p_t\Bigg(\abs{\hat{h}(p)}^{2}\abs{\rho_{t\bar{w}}(p)}^{2}+\frac{\partial\hat{h}}{\partial t}(p)\left( \rho_{t\bar{w}}(p)\rho_{w}(p)+\rho_{tw}(p)\rho_{\bar w}(p)\right)+\\
  \left(\frac{\partial\hat{h}}{\partial t}(p)\right)^{2}\abs{\rho_{w}(p)}^{2}\Bigg).\end{multline*}
  	From \eqref{eq:phi} and \eqref{eq:deltaw}, this can be simplified to
  	\[
  		\abs{c(p)}^2=\tau^{2}p_t\left(\abs{\hat{h}(p)}^{2}\abs{\rho_{t\bar{w}}(p)}^{2}+\frac{1}{4}\left(\frac{\partial\hat{h}}{\partial t}(p)\right)^{2}\right).\]
Computing $a(p)b(p)$, we have
  		\begin{multline*}
  a(p)b(p)=\frac{\hat{h}(p)\tau(1-\tau)}{4} \left(4\hat{h}(p)\tau\abs{\rho_{tw}(p)}^{2}p_t-\frac{\partial\hat{h}}{\partial t}(p)-\frac{\partial^{2}\hat{h}}{\partial t^{2}}(p)p_t\right)\\
  		=\hat{h}^{2}(p)\tau^{2}(1-\tau)\abs{\rho_{tw}(p)}^{2} p_t-\tau(1-\tau)\left(\frac{ \bd \hat{h}}{\bd t}(p)\frac{\hat{h}(p) }{4} +\frac{\hat h(p)}{4}
\frac{ \bd ^{2}\hat{h}}{\bd t^{2}}(p) p_t \right).\end{multline*}
Combining these computations, we obtain
 \begin{multline*}
 a(p)b(p)-\abs{ c(p)}^2=
 -(\hat{h}(p))^{2}\tau^{3}\abs{\rho_{tw}(p)}^{2} p_t\\
 -\tau(1-\tau)\left(\frac{ \bd \hat{h}}{\bd t}(p)\frac{\hat{h}(p) }{4} +\frac{\hat h(p)}{4}
\frac{ \bd ^{2}\hat{h}}{\bd t^{2}}(p)p_t \right)-\tau^{2}p_t\frac{1}{4}\left(\frac{\partial\hat{h}}{\partial t}(p)\right)^{2}.
\end{multline*}
To linearize this expression, we set $g=(\hat h)^{1/(1-\tau)}$.  Since $\hat h>0$, we may assume that $g$ is also real and positive.
Substituting $\hat{h}=g^{(1-\tau)}$, we get
\begin{multline*}
 a(p)b(p)-\abs{c(p)}^2=\\
 \left(-\tau^3 p_t g(p)\abs{\rho_{tw}(p)}^2-\frac{ 1}{4}\tau(1-\tau)^2\left(\frac{ \bd g}{\bd t}(p)+p_t\frac{ \bd^2 g}{\bd t^2}(p)\right)\right)(g(p))^{1-2\tau}.
\end{multline*}
Hence, $a(p)b(p)-\abs{c(p)}^2\geq 0$ implies
\begin{equation} \label{eq:tau1}
  -\frac{ 1}{4}\tau(1-\tau)^2 \left(\frac{ \bd g}{\bd t}(p)+p_t\frac{ \bd^2 g}{\bd t^2}(p)\right)-\tau^3 p_t g(p) \abs{ \rho_{tw}(p)}^2\geq 0.
\end{equation}

Our hypothesis on $\abs{\frac{\bd^2\tilde{\delta}}{\bd w\bd\bar{z}}}$ and \eqref{eq:bar_w_z_derivative} imply
\[\abs{\rho_{tw}(p) }>\frac{C}{p_t}\]
on $M$.  Coupled with \eqref{eq:tau1}, this implies that
\begin{equation}\label{eq:tntn1}
 -\frac{ 1}{4}\tau(1-\tau)^2 \left(\frac{ \bd g}{\bd t}+t\frac{ \bd^2 g}{\bd t^2}\right)-\frac{C^2}{t^2}\tau^3 tg> 0
 \end{equation}
 when $A\leq t\leq B$.  Note that we may suppress the dependency on $p$ since $p_w$ and $\arg p_z$ no longer play a role in this inequality.

We will show that this implies a contradiction unless $\tau<\frac{ \pi}{2 C \log \frac{ B}{A}+\pi}$. Let $\tilde{ g}(s,w)=g(e^s,w)$, i.e., we will use the substitution $t=e^s$ so that $\log A\leq s\leq\log B $ on $M$.  Then
\[\frac{\partial\tilde{ g}}{\partial s}(s,w)=e^s\frac{\partial g}{\partial t}(e^s,w)\]
and
\[\frac{\partial^2\tilde{ g}}{\partial s^2}(s,w)=e^s \frac{\partial g}{\partial t}(e^s,w)+e^{2s}\frac{\partial^2 g}{\partial t^2}(e^s,w)=t\left(\frac{ \bd g}{\bd t}(t,w)+t\frac{ \bd^2 g}{\bd t^2}(t,w)\right).\]
Henceforth we fix $w$ and treat $\tilde g$ as a function of the single variable $s$.  Substituting in \eqref{eq:tntn1}, we have
\begin{equation}\label{eq:off1}
	\frac{ 1}{4}\tau(1-\tau)^2 \frac{ d^2\tilde{ g} }{d s^2}+\tau^3C^2\tilde{ g} < 0.
\end{equation}
Let us assume we have a strictly positive function $\tilde{ g}$ on some interval satisfying \eqref{eq:off1}. After making the substitution $u=\sqrt{\frac{ \tau^3 C^2}{\frac{ 1}{4}\tau(1-\tau)^2}}s=\frac{2\tau C}{1-\tau}s$, we have
\begin{equation} \label{eq:seva1}
\frac{ d^2\til{g}}{du^2}+\tilde{g}<0.
\end{equation}
when $\frac{2\tau C}{(1-\tau)}\log A\leq u\leq\frac{2\tau C}{(1-\tau)}\log B$.

Suppose $\tau\geq\frac{ \pi}{2 C \log \frac{ B}{A}+\pi}$.  Then $ \frac{2\tau  C}{(1-\tau)}\log\frac{B}{A}\geq \pi$, so \eqref{eq:seva1} holds on an interval of length $\pi$.  Since $\til g>0$, \eqref{eq:seva1} implies $\frac{d^2 \tilde{ g}}{d u^2}<0$, so $\til{g}$ is strictly concave down on this interval. Therefore, there must be an interval of length $\frac{\p}{2}$ on which $\tilde g$ is strictly increasing or strictly decreasing. If it is strictly increasing, then we can flip it upward using a reflection to make it strictly decreasing and positive. So, after a translation in $u$, $\frac{ d\tilde{ g}}{du}<0\; \text{ on }[0,\frac{ \pi}{2}]$.
In \cite{DiFo77a}, Diederich and  Forn{\ae}ss have shown in the proof of Theorem 6 that \eqref{eq:seva1} has no positive, strictly decreasing solution on $[0,\frac{\pi}{2}]$.  Thus, we have a contradiction, and the conclusion of the theorem follows.
\end{proof}
\subsection{Lower bound for Diederich-Forn{\ae}ss Index  }

In this section we will prove Theorem \ref{tctokh}.  Since the weakly pseudoconvex points may be divided into multiple connected components, we will begin with some technical lemmas that will allow us to deal with each connected component separately and patch the end results together.

 	\begin{lem}\label{sono}
 		Let $\Omega $ be a Hartogs domain in $\mathbb{C}^2$, and let $K$ be a connected component of $\left\{p\in b\Omega:\frac{\partial\tilde\delta}{\partial z}(p)=0\text{ and }\frac{\partial^2\tilde\delta}{\partial z\partial\bar z}(p)=0\right\}$ that is disk-like.  Then there exists a real valued  function $u$ defined in a neighborhood of $K$ such that
 		\begin{equation}\label{tona}
 			 \frac{\bd u}{\bd z}=-\frac{\bd^2\tilde{\delta}}{\bd z\bd\bar{w}}/\frac{\bd\tilde{\delta}}{\bd\bar{w}}
 \end{equation}
 		on $K$ and the complex Hessian of $\rho=\tilde{\delta} e^u $ is positive semi-definite on $K$.
 	\end{lem}
 	\begin{proof}
 		On $K$, we have $\frac{\bd\tilde{\delta}}{\bd z}=0$, so \eqref{eq:deltaw} applies.  In particular, $\left(\frac{\bd\tilde{\delta}}{\bd\bar{w}}\right)^{-1}=4\frac{\bd\tilde{\delta}}{\bd w}$ on $K$.  Hence, if we define
 		\begin{equation}\label{alfi}
 		 \alpha=-\frac{\bd^2\tilde{\delta}}{\bd z\bd\bar{w}}
 		/\frac{\bd\tilde{\delta}}{\bd\bar{w}}dz-\frac{\bd^2\tilde{\delta}}{\bd \bar{ z}\bd w}
 		/\frac{\bd\tilde{\delta}}{\bd w}d\bar{z}, \end{equation}
 then this is a scalar multiple of D'Angelo's one-form $\alpha$ (see \cite{BoSt93} or Section 5.9 in \cite{Str10} for detailed analysis of this form).  Fix an analytic disk $M\subset K$.  Lemma 5.14 in \cite{Str10} implies that $\alpha$ is closed on $M$.  Since $M$ is simply connected, there exists a smooth function $\tilde{u}$ on $M$ such that $d_M\tilde{u}=\alpha$ on $M$.  If we add the requirement that $\tilde u$ is equal to zero at the center of the disk (i.e., the unique point in $M$ at which $z=0$), then $\tilde u$ is uniquely determined.  Hence, we may solve this equation for every analytic disk in $K$ to obtain a smooth function $\tilde u$ satisfying \eqref{tona} on $K$.  We extend $\tilde u$ smoothly to a neighborhood of $K$. Let $u=\tilde{u}+s\tilde{\delta}$ for some number $s>0$ to be chosen later.  On $K$, we have $\frac{\partial\tilde\delta}{\partial z}=0$, so $\frac{\partial u}{\partial z}=\frac{\partial\tilde u}{\partial z}$, and hence \eqref{tona} is satisfied.  Furthermore, $\frac{\partial^2\tilde\delta}{\partial z\partial\bar z}=0$ on $K$, so  $\rho=\tilde{\delta}e^u$ satisfies $\frac{\partial^2\rho}{\partial z\partial\bar z}=0$ on $K$.  Further computation gives us
 \[
   \frac{\partial^2\rho}{\partial z\partial\bar w}=e^u\left(\frac{\bd^2\tilde{\delta}}{\bd z\bd\bar{w}}+\frac{\bd u}{\bd z}\frac{\bd\tilde{\delta}}{\bd\bar{w}}\right)
 \]
 on $K$ and
 \[
   \frac{\partial^2\rho}{\partial w\partial\bar w}=e^u\left(\frac{\bd^2\tilde{\delta}}{\bd w\bd\bar{w}}+\frac{\bd u}{\bd w}\frac{\bd\tilde{\delta}}{\bd\bar{w}}+\frac{\bd\tilde{\delta}}{\bd w}\frac{\bd u}{\bd\bar{w}}\right)
 \]
 on $K$.  Using \eqref{tona}, $\frac{\partial^2\rho}{\partial z\partial\bar w}=0$ on $K$, while \eqref{eq:deltaw} gives us
 \[
   \frac{\partial^2\rho}{\partial w\partial\bar w}=e^u\left(\frac{\bd^2\tilde{\delta}}{\bd w\bd\bar{w}}+\frac{\bd\tilde u}{\bd w}\frac{\bd\tilde{\delta}}{\bd\bar{w}}+\frac{\bd\tilde{\delta}}{\bd w}\frac{\bd\tilde u}{\bd\bar{w}}+\frac{s}{2}\right)
 \]
 on $K$.  For $s$ sufficiently large, we will have $\frac{\partial^2\rho}{\partial w\partial\bar w}>0$ on $K$, so the complex Hessian of $\rho$ will be positive semi-definite on $K$.
 	\end{proof}
 	\begin{lem}\label{fono} 		
 		Let $\Omega$ be a Hartogs domain in $\mathbb{C}^2$. If   $K\subset b\Omega$ is compact, and $\frac{\bd\tilde{\delta}}{\bd z}\neq 0$ on $K$, then there exists a defining function for $\Omega$ in a neighborhood of $K$ which is plurisubharmonic on $b\Omega$ near $K$. 		
 	\end{lem}
 	\begin{proof}
 		Let $p\in K$.  Suppose $p=(0,w)$ for some $w\in\mathbb{C}$.  Let $\{(z_j,w_j)\}$ be any sequence in $b\Omega$ converging to $(0,w)$.  By restricting to a subsequence, we may assume $\left\{\frac{(z_j,w_j-w)}{\sqrt{|z_j|^2+|w_j-w|^2}}\right\}$ converges to a unit length vector $(u_1,u_2)$ tangential to $b\Omega$ at $p$.  We note that every tangent vector may be obtained in this way (in fact, \cite{Fed59} takes this as the definition of a tangent vector, which agrees with the usual definition on domains with $C^1$ boundaries). Since $\Omega$ is Hartogs, $(e^{i\theta}z_j,w_j)\in b\Omega$ for any $\theta\in\mathbb{R}$, so $(e^{i\theta}u_1,u_2)$ is tangential to $b\Omega$ at $p$ for any $\theta\in\mathbb{R}$.  Hence, $e^{i\theta}u_1\frac{\partial\tilde\delta}{\partial\bar z}(p)+u_2\frac{\partial\tilde\delta}{\partial\bar w}(p)=0$ for any $\theta\in\mathbb{R}$.  Since $\frac{\partial\tilde\delta}{\partial\bar z}(p)\neq 0$, we must have $u_1=0$.  However, this means that every tangent vector at $p$ is of the form $(0,u_2)$, which is inadequate to span a tangent space of $3$ real dimensions.  Therefore, we may assume that $z\neq 0$ whenever $(z,w)\in K$.

 Since $\Omega$ is Hartogs, $\tilde\delta$ depends only on $|z|^2$ and $w$, so in a neighborhood of any point $p\in b\Omega$ for which $\frac{\partial\tilde\delta}{\partial z}(p)\neq 0$, we can use the implicit function theorem to find a strictly positive smooth function $f_p(w)$ such that $(z,w)\in b\Omega$ whenever $|z|^2=f_p(w)$.
 	
 		 Near $p$, our defining function will be
 		$$\rho_p(z,w)=(f_p(w))^{-1}|z|^2-1.$$  The implicit function theorem does not guarantee that $f_p(w)=f_q(w)$ for all $p,q\in b\Omega$ satisfying $\frac{\bd\tilde{\delta}}{\bd z}\neq 0$, but for $q$ sufficiently close to $p$ we do have $f_p(w)=f_q(w)$, so that $\rho_p$ can be extended to a global function on a neighborhood of $K$.  Henceforth, we will suppress the subscript $p$ when writing $f$ and $\rho$.
 		On a neighborhood of $K$, we may compute the first derivatives
 \[
   \frac{\partial\rho}{\partial z}(z,w)=(f(w))^{-1}\bar z\text{ and }\frac{\partial\rho}{\partial w}(z,w)=-(f(w))^{-2}\frac{\partial f}{\partial w}(w)|z|^2,
 \]
 and the second derivatives
 \[
   \frac{\bd^2\rho}{\bd z\bd \bar{z}}(z,w)=(f(w))^{-1}, \frac{\bd^2\rho}{\bd w\bd\bar{z}}(z,w)=-(f(w))^{-2}\frac{\bd f}{\bd w}(w)z,
 \]
 and
 \[
   \frac{\bd^2\rho}{\bd w\bd\bar{w}}(z,w)=-(f(w))^{-2}\frac{\bd^2 f}{\bd w\bd\bar{w}}(w)|z|^2+2(f(w))^{-3}\abs{\frac{\bd f}{\bd w}(w)}^2|z|^2.
 \]
 For $(p_z,p_w)$ in some neighborhood of $K$ in $b\Omega$, we have $|p_z|^2=f(p_w)$.
 			Since $(\xi,\eta)=\left(\frac{\partial f}{\partial w}(p_w),\bar p_z\right)$ spans the complex tangent space at $(p_z,p_w)$, the Levi form at $(p_z,p_w)$ in this direction is given by
 \begin{align*}
   H_\rho(|p_z|^2,p_w,\xi,\eta)&=\frac{\bd^2 \rho}{\bd z\bd\bar{z}}(p_z,p_w)\abs{\frac{\bd f}{\bd w}(p_w)}^2+\frac{\bd^2\rho}{\bd\bar{z}\bd w}(p_z,p_w)\bar p_z\frac{ \bd f}{\bd \bar{w}}(p_w)\\
   &+\frac{\bd^2\rho}{\bd z\bd\bar{w}}(p_z,p_w)z\frac{\bd f}{\bd w}(p_w)+\frac{\bd^2\rho}{\bd w\bd\bar{w}}(p_z,p_w)f(p_w)\\
   &=(f(p_w))^{-1}\abs{\frac{\bd f}{\bd w}(p_w)}^2-\frac{\bd^2 f}{\bd w\bd\bar{w}}(p_w).
 \end{align*}
 			From the pseudoconvexity of $\Omega$, $H_\rho(|p_z|^2,p_w,\xi,\eta)\geq 0$ on $b\Omega$ near $K$. This implies $$\abs{\frac{\bd f}{\bd w}(p_w)}^2\geq f(p_w)\frac{\bd^2 f}{\bd w\bd\bar{w}}(p_w).$$
 For $(p_z,p_w)$ on $b\Omega$ near $K$, the determinant of the complex Hessian of $\rho$ is given by
 \begin{multline*}
   \frac{\bd^2\rho}{\bd z\bd \bar{z}}(p_z,p_w)\frac{\bd^2\rho}{\bd w\bd\bar{w}}(p_z,p_w)-\abs{\frac{\bd^2\rho}{\bd w\bd\bar{z}}(p_z,p_w)}^2=\\
   -(f(p_w))^{-2}\frac{\bd^2 f}{\bd w\bd\bar{w}}(p_w)+(f(p_w))^{-3}\abs{\frac{\bd f}{\bd w}(p_w)}^2.
 \end{multline*}
 This is nonnegative since the Levi form is nonnegative.  Hence, $\rho$ is a defining function that is plurisubharmonic on $b\Omega$ near $K$.
 	\end{proof}
 	
 	\begin{lem}\label{kono}
 		Let $\Omega$ be a smooth, bounded, pseudoconvex domain in $\mathbb{C}^2 $, and let $K$ denoted the set of weakly pseudoconvex points in $b\Omega$. 
 		Let $U_1$ be a neighbourhood of $K$ (not necessarily connected) $U_2$ an open set such that $K\cap\overline{U_2}=\emptyset$ and $b\Omega\subset U_1\cup U_2$.  Let $\rho_1$ be a smooth defining function for $\Omega$ on  $U_1$.  Then, for every $0<\tau_3<1$, there exists a neighborhood $U_3$ of $b\Omega$ such that $\overline{U_3}\subset U_1\cup U_2$ and a smooth defining function $\rho_3$ for $\Omega$ on $U_3$ such that $$\rho_3=\rho_1\text { on }U_1\setminus U_2$$
 and
 		$$ i\bd\dbar(-(-\rho_3)^{\tau_3})\geq iM_3(-\rho_3)^{\tau_3}\bd\dbar\abs{z}^2\text{ on }U_3\cap\Omega.$$
 	\end{lem}
 	\begin{proof}
 		
 		Let $U_3$ be a neighborhood of $b\Omega$ such that $\overline{U_3}\subset U_1\cup U_2$.  Fix $\chi\in C^{\infty}(\mathbb{C}^2)$ such that $0\leq\chi\leq 1$, $\chi\equiv 1$ on a neighborhood of $\overline{U_3\setminus U_2}$, and $\chi\equiv 0 $ on a neighborhood of $\overline{U_3\setminus U_1}$.  Since $\Omega$ is strictly pseudoconvex on $\overline{U_2}$, we may let $\rho_2$ be a strictly plurisubharmonic defining function for $\Omega$ on $U_2\cap U_3$ (this may require shrinking $U_3$).  Set $\rho_3=\chi\rho_1+(1-\chi)\rho_2$.  We know $e^{\lambda_3\rho_3}-1$ is strictly plurisubharmonic on $\overline{U_3\cap U_1\cap U_2}$ for $\lambda_3>0$ sufficiently large (see, for example, Theorem 3.4.4 \cite{ChSh01}).
 		More precisely, there exists some $N_3>0$ such that $$i\bd\dbar(e^{\lambda_3\rho_3}-1)>iN_3\bd\bar{\bd}\abs{z}^2$$ on $\overline{U_3\cap U_1\cap U_2}$.  We compute
 		$$ i\bd\bar{\bd}(e^{\lambda_3\rho_3}-1)=i\lambda_3e^{\lambda_3\rho_3}(\bd\bar{\bd}\rho_3+\lambda_3\bd\rho_3\wedge\bar{\bd}\rho_3 )  \geq i N_3 \bd\bar{\bd}\abs{z}^2.$$
 		Hence,
 		\begin{equation}\label{phero1}
 		i\bd\bar{\bd}\rho_3\geq i e^{-\lambda_3\rho_3}\frac{N_3}{\lambda_3}\bd\bar{\bd}\abs{z}^2-i\lambda_3\bd\rho_3\wedge\bar{\bd}\rho_3
 		\end{equation}
on $\overline{U_3\cap U_1\cap U_2}$. 		On the other hand, we want to show $$i\bd\bar{\bd}(-(-\rho_3)^{\tau_3})\geq i M_3(-\rho_3)^{\tau_3}\bd\bar{\bd}\abs{z}^2  $$
 		on $U_3\cap\overline{ U_1\cap U_2}\cap\Omega$.  Expanding the left hand side, this is equivalent to $$i\tau_3(1-\tau_3)(-\rho_3)^{\tau_3-2}\bd\rho_3\wedge\bar{\bd}\rho_3+i\tau_3(-\rho_3)^{\tau_3-1}\bd\bar{\bd}\rho_3\geq i M_3(-\rho_3)^
 		{\tau_3}\bd\bar{\bd}\abs{z}^2  $$
 	on $U_3\cap\overline{ U_1\cap U_2}\cap\Omega$.  To show this, we need
 		\begin{equation}\label{phero2}
 		i\bd\bar{\bd}\rho_3\geq i\frac{M_3}{\tau_3}(-\rho_3)\bd\bar{\bd}\abs{z}^2-i(1-\tau_3)(-\rho_3)^{-1} \bd\rho_3\wedge\bar{\bd}\rho_3
 		\end{equation}
 		on $U_3\cap\overline{ U_1\cap U_2}\cap\Omega$.  Since $\rho_3$ is close to zero near the boundary, the first term on the right hand side of \eqref{phero1} is greater than the first term on the right hand side of\eqref{phero2}. Similarly, since $0<\tau_3<1,$ the second term on the right hand side of \eqref{phero1} bounds the second term on the right hand side of \eqref{phero2} when $\rho_3$ is sufficiently close to zero.  Hence, we may shrink $U_3$ sufficiently small so that \eqref{phero2} holds.

 	\end{proof}


With these tools in place, we are finally ready to prove our main theorem for this section.

\begin{proof}[Proof of Theorem \ref{tctokh}]
First we will consider the possibility that $K$ is an annulus-like connected component of $\left\{p\in b\Omega:\frac{\partial\tilde\delta}{\partial z}(p)=0\text{ and }\frac{\partial^2\tilde\delta}{\partial z\partial\bar z}(p)=0\right\}$.
By assumption, there exist constants $B\geq A>0$ and $C>0$ such that $A\leq|z|^2\leq B$ on $K$, $\abs{\frac{\partial\tilde\delta}{\partial w\partial\bar z}}\leq\frac{C}{|z|}$ on $K$, and $\frac{ 2\tau C}{1-\tau}\log B- \frac{ 2\tau C}{1-\tau}\log  A <\pi$.  By this last hypothesis, there exists a constant $\phi$  such that $$ \sin\left(\frac{ 2\tau C}{1-\tau}\log t+\phi\right)>0$$ on $A\leq t\leq B.$  There then exists $\eps>0$ such that a positive  solution of
 \begin{equation}\label{ cin}
 -\frac{ 1}{4}\tau(1-\tau)^2 \left(\frac{ \bd g}{\bd t}(t)+t\frac{ \bd^2 g}{\bd t^2}(t)\right)-\frac{C^2}{t^2}\tau^3 tg(t)\geq \eps \tau^3t\frac{C^2}{t^2}
 \end{equation} on $A\leq t\leq B $ is given by
    $g(t)=c_1\cos(\frac{ 2\tau C}{1-\tau}\log t)+c_2\sin(\frac{ 2\tau C}{1-\tau}\log t) -\eps$, where $c_1=\sin\phi$ and $c_2=\cos\phi$.
     Then \eqref{ cin} implies
\begin{equation}\label{tno}
-\frac{ 1}{4}\tau(1-\tau)^2 \left(\frac{ \bd g}{\bd t}(t)+t\frac{ \bd^2 g}{\bd t^2}(t)\right)>\tau^3 tg(t)\frac{C^2}{t^2} .
\end{equation}
Our hypotheses now imply that $g$ satisfies \eqref{eq:tau1}.

 Define $h(z,w)=(g(|z|^2))^{1-\tau}, $ and $\sigma=-h(-\tilde{\delta})^\tau$.  Following the proof of Theorem \ref{main} backward from \eqref{eq:tau1}, we see that $H_{\sigma}\geq 0 $ on $U\cap\Omega$, where $U$ is some neighborhood of $K$ (see similar arguments in \cite{Har18} and
 \cite{Liu17a}). Therefore, $\sigma$ is plurisubharmonic in a neighborhood of $U\cap\Omega$.

  Next, suppose $K$ is a disk-like connected component of\\ $\left\{p\in b\Omega:\frac{\partial\tilde\delta}{\partial z}(p)=0\text{ and }\frac{\partial^2\tilde\delta}{\partial z\partial\bar z}(p)=0\right\}$.  We have shown in Lemma  \ref{sono} that there exists a real-valued function $u$ satisfying \eqref{tona} on $K$. Then on $K$,
 \begin{equation}\label{sin}
  \frac{\bd u}{\bd z}=-\frac{\bd^2\tilde{\delta}}{\bd z\bd\bar{w}}/\frac{\bd\tilde{\delta}}{\bd\bar{w}}=-4\frac{\bd\til{\delta}}{\bd w}\frac{\bd^2\til{\delta}}{\bd z\bd\bar{w}}.
  \\
   \end{equation}
   Since $\frac{\partial}{\partial\bar z}$ is tangential to $K$, we may differentiate both sides of \eqref{sin} to obtain
   \[\frac{\bd^2 u}{\bd\bar{z}\bd z}=-4\frac{\bd\tilde{\delta}}{\bd w}\frac{\bd^3\til{\delta}}{\bd\bar{z}\bd\bar{w}\bd z}-4\frac{\bd^2\til{\delta}}{\bd\bar{z}\bd w}\frac{\bd^2\til{\delta}}{\bd z\bd\bar{w}}\]
   on $K$.  Since $u$ is real-valued, this can be written
   \begin{multline*}
   \frac{\bd^2 u}{\bd\bar{z}\bd z}=-4\re\left(\frac{\bd\tilde{\delta}}{\bd w}\frac{\bd^3\til{\delta}}{\bd\bar{z}\bd\bar{w}\bd z}\right)-4\abs{\frac{\bd^2\til{\delta}}{\bd z\bd\bar{w}}}^2\\
   =-2\frac{\bd^2}{\bd\bar{z}\bd z}\abs{\frac{\bd\tilde{\delta}}{\bd w}}^2+2\abs{\frac{\bd^2\til{\delta}}{\bd z\bd\bar{w}}}^2+2\abs{\frac{\bd^2\til{\delta}}{\bd z\bd w}}^2-4\abs{\frac{\bd^2\til{\delta}}{\bd z\bd\bar{w}}}^2.
   \end{multline*}
   Since \eqref{eq:deltaw} implies that $\abs{\frac{\bd\tilde{\delta}}{\bd w}}^2$ is constant and equations \eqref{eq:wz_derivative} and \eqref{eq:bar_w_z_derivative} imply that $\abs{\frac{\bd^2\til{\delta}}{\bd z\bd\bar{w}}}^2=\abs{\frac{\bd^2\til{\delta}}{\bd z\bd w}}^2$, we are left with
   \[\frac{\bd^2 u}{\bd\bar{z}\bd z} =0\]
   on $K$.  Let $\hat{h}(z,w)=e^{\tau u(z,w)-s\abs{z}^2}$ for $s>0$ to be chosen later.  Let $p\in\Omega$ sufficiently close to the boundary of $\Omega$ satisfy $\pi(p)\in K$.  As in the proof of Theorem \ref{main} we set
 \[
   a(p)=4\hat{h}(p)\tau\abs{ \rho_{tw}(\pi(p))}^2 |p_z|^2-\frac{ \bd^2\hat{ h}}{\bd z\bd\overline{ z}}(p),
   \]
   so that
   \begin{multline*}
    a(p)=e^{\tau u(p)-s\abs{p_z}^2}4\tau \abs{\rho_{t w}(\pi(p))}^2 |p_z|^2\\-e^{\tau u(p)-s\abs{p_z}^2}\abs{\tau\frac{\bd u}{\bd z}(p)-s\overline{p_z}}^2-e^{\tau u(p)-s\abs{p_z}^2}\left(\tau\frac{ \bd^2 u}{\bd z\bd\bar{z}}(p)-s\right).
   \end{multline*}
   Notice that \eqref{sin}, \eqref{eq:bar_w_z_derivative}, and \eqref{eq:phi} imply
   \begin{multline*}
   \re\left(\frac{\bd u}{\bd z}(\pi(p))(\pi(p))_z\right)=-4\re\left(\frac{\bd\tilde{\delta}}{\bd w}(\pi(p))\frac{\bd^2\tilde{\delta}}{\bd z\bd\bar{w}}(\pi(p))(\pi(p))_z\right)\\
   =-4|(\pi(p))_z|^2\re(\rho_{w}(\pi(p))\rho_{t\bar{w}}(\pi(p)))=0. \end{multline*}
   Hence, using \eqref{eq:deltaw}, \eqref{sin}, and \eqref{eq:bar_w_z_derivative}, we have
   \begin{multline*}
     \abs{\tau\frac{\bd u}{\bd z}(\pi(p))-s\overline{(\pi(p))_z}}^2=\tau^2\abs{\frac{\bd u}{\bd z}(\pi(p))}^2+s^2|(\pi(p))_z|^2\\
     =4\tau^2|(\pi(p))_z|^2|\rho_{tw}(\pi(p))|^2+s^2|(\pi(p))_z|^2.
   \end{multline*}
   Since $u$ is at least $C^2$ and $\abs{\pi(p)-p}=\delta(p)$, we have
   \[
     \abs{\tau\frac{\bd u}{\bd z}(p)-s\overline{p_z}}^2\leq 4\tau^2|p_z|^2|\rho_{tw}(\pi(p))|^2+s^2|p_z|^2+O(-\tilde\delta(p)).
   \]
   Furthermore,
   \[
     \frac{ \bd^2 u}{\bd z\bd\bar{z}}(p)\leq\frac{ \bd^2 u}{\bd z\bd\bar{z}}(\pi(p))+O(-\tilde\delta(p)),
   \]
   so since $\frac{ \bd^2 u}{\bd z\bd\bar{z}}(\pi(p))=0$, this allows us to simplify $a$ as follows:
   \[
    a(p)\geq e^{\tau u(p)-s\abs{p_z}^2}\left(s+4\tau(1-\tau)|p_z|^2|\rho_{tw}(\pi(p))|^2-s^2|p_z|^2\right)-O(-\tilde\delta(p)).
   \]

  Also as in the proof of Theorem \ref{main}, we set
   \[ b(p)=\frac{ 1}{4} \hat{h}(p)\tau(1-\tau)=\frac{1}{4}e^{\tau u(p)-s\abs{p_z}^2}\tau(1-\tau) \]
   and
   \[
     c(p)=\hat{h}(p)\tau\rho_{t\overline{ w}}(\pi(p))\overline{p_z}+\tau\frac{ \bd\hat{ h}}{\bd z}(p)\frac{ \bd\til{\delta}}{\bd\bar{w}}(p).
   \]
   Using \eqref{sin} and \eqref{eq:bar_w_z_derivative}, we may simplify and obtain
   \[
   	\abs{c(p)}\leq\tau(1-\tau) e^{\tau u(p)-s\abs{p_z}^2}\rho_{t\overline{ w}}(\pi(p))\overline{p_z}-s\tau e^{\tau u(p)-s\abs{p_z}^2}\overline{p_z}\frac{\bd\til{\delta}}{\bd\bar{w}}(p)+O(-\tilde\delta(p)). \]
   Hence, using \eqref{eq:deltaw} and \eqref{eq:phi},
   \begin{multline*}
     a(p)b(p)-|c(p)|^2\geq\\
     \frac{1}{4}\tau(1-\tau)e^{2\tau u(p)-2s\abs{p_z}^2}\left(s-s^2|p_z|^2\right)-\frac{1}{4}\tau^2e^{2\tau u(p)-2s|p_z|^2}s^2|p_z|^2-O(-\tilde\delta(p)).
   \end{multline*}
   Since $b(p)>0$, the matrix $\begin{pmatrix}a(p)&c(p)\\\overline{c(p)}&b(p)\end{pmatrix}$ is positive definite provided that $a(p)b(p)-|c(p)|^2>0$, and this is true if $1-\tau>s|p_z|^2$ and $-\tilde\delta(p)$ is sufficiently small.  We assume that $|p_z|^2\leq B$ on $K$.  Hence, $a(p)b(p)-|c(p)|^2>0$ if $s<\frac{1-\tau}{B}$ and $-\tilde\delta(p)$ is sufficiently small.  As a result, $H_{\sigma}\geq 0$ on $U\cap\Omega$ for some neighborhood $U$ of $K$.

   Finally, suppose $K$ is a set of weakly pseudoconvex points satisfying the hypotheses of Lemma \ref{fono}.  Then Lemma \ref{fono} gives us a defining function which is plurisubharmonic on $K$.  Following the construction of \cite{FoHe07}, we may construct a defining function with a Diederich-Forn{\ae}ss Index of $\tau$ near $K$.

   Finally, we decompose the weakly pseudoconvex points into $\{K_j\}$ where each $K_j $ is either disk-like, annulus-like or satisfies the hypotheses of Lemma \ref{fono}.  Let $U_1$ be a neighborhood of the weakly pseudoconvex points such that each connected component of $U_1$ contains exactly one component $K_j$.  Lemma \ref{kono} shows that there exists an appropriate defining function defined on some neighborhood of $b\Omega$.

\end{proof}


As an important special case, we now show that our results are sharp on the worm domain.

  \begin{proof}[Proof of Corollary \ref{cor:worm_domain}]
  	
  	Let $M_r=\{(z,0):1\leq|z|\leq r\}$ denote the annulus containing all of the weakly pseudoconvex points in the boundary of $\Omega_r$.  Let $\rho_r$ be the defining function given by Definition \ref{worm}.  Since $r$ is taken to be fixed, we will omit the subscript $r$ in the following.  For $(z,w)\in M,$ we have 	
  	\begin{align}
\nonumber  			\frac{ \bd \rho}{\bd w}(z,w)&=e^{-i\log\abs{z}^2},\\
  \label{eq:real}\frac{\bd^2 \rho}{\bd \bar{z}\bd w}(z,w)&= \frac{-iz}{\abs{z}^2} e^{-i \log\abs{z}^2}\text{, and }\\
\nonumber  			\frac{\bd^2 \rho}{\bd z\bd w}(z,w)&= \frac{-i\bar{z}}{\abs{z}^2} e^{-i \log\abs{z}^2}.
  		\end{align}
  		We also have
  		$ \frac{\bd \rho}{\bd z}(z,w)=0.$ Let $w=u+iv$, and $z=x+iy$.  Since
  		\[\frac{\bd\rho}{\bd w}(z,w)=\frac{1}{2}\frac{\bd \rho}{\bd u}(z,w)-i\frac{1}{2}\frac{\bd \rho }{\bd v}(z,w)=\cos(\log\abs{z}^2)-i\sin(\log\abs{z}^2),\]
  		the real normal vector will be
  		$$\grad \rho(z,w)=(0,0,2\cos(\log\abs{z}^2), 2\sin(\log\abs{z}^2)),$$
  		and the real tangent space is spanned by
  		$T_1=(1,0,0,0)$, $T_2=(0,1,0,0)$, and  $T_3=(0,0,-\sin(\log\abs{z}^2),\cos(\log\abs{z}^2))$.
  		We denote $\partial_\nu=\frac{1}{2}\bigtriangledown\rho\cdot \bigtriangledown$ and $\partial_3=T_3\cdot\nabla$.  In this notation,
  		\[\frac{\bd}{\bd w}=\frac{1}{2}e^{-i\log\abs{z}^2}(\partial_\nu-i\partial_3).\]
  		Hence,
  		\[\frac{\bd^2 \til{\delta}}{\bd\bar{z}\bd w}(z,w)=\frac{1}{2}e^{-i\log\abs{z}^2}\left(\partial_\nu\frac{\bd \til{\delta}}{\bd\bar{z}}-i\partial_3\frac{\bd \til{\delta}}{\bd\bar{z}}\right).\]
  		Since $\partial_3$ and $\frac{\bd}{\bd\bar{z}}$ are tangential, $\partial_3\frac{\bd\tilde{\delta}}{\bd\bar{z}}(z,w)=\abs{\bigtriangledown\rho(z,w)}^{-1}\partial_3\frac{\bd \rho}{\bd\bar{z}}(z,w)$ (see, for example, the discussion preceding (2.9) in \cite{HaRa13} for justification).
  		 Using \eqref{eq:gradd} to differentiate the gradient in the normal direction, we obtain $\partial_\nu\frac{\bd\til{\delta}}{\bd\bar{z}}(z,w)=0$. We get
  		\[\frac{\bd^2\til{\delta}}{\bd\bar{z}\bd w}(z,w)=-i\frac{1}{2}e^{-i\log\abs{z}^2}\abs{\grad\rho(z,w)}^{-1}\partial_3\frac{\bd\rho}{\bd\bar{z}}(z,w).\]

  		On the other hand,
  \[
    \partial_3=i e^{i\log\abs{z}^2}\frac{\bd}{\bd w}-i e^{-i\log\abs{z}^2}\frac{\bd}{\bd\bar w},
  \]
  		so
  		\[\partial_3\frac{\partial\rho}{\partial\bar z}(z,w)=ie^{i\log\abs{z}^2}\frac{\bd^2\rho}{\bd w\bd\bar{z}}(z,w)-ie^{-i\log\abs{z}^2}\frac{\bd^2\rho}{\bd\bar{w} \bd\bar{z}}(z,w).\]
  		Substituting \eqref{eq:real}, we compute
  		\[\frac{\bd^2\til\delta}{\bd\bar{z}\bd w}(z,w)=\frac{i}{4}e^{-i\log\abs{z}^2}\frac{ 2z}{\abs{z}^2}, \] so	
  		\[\abs{\frac{\bd^2\til\delta}{\bd\bar{z}\bd w}(z,w) }=\frac{1}{2}\frac{ 1}{\abs{z}}.\]
  		From Theorem \ref{main}, the Diederich-Forn{\ae}ss Index is less than $\frac{ \pi}{2C \log  r^2+\pi}$, for every $C<\frac{1}{2}$, so the Diederich-Forn{\ae}ss Index is less than or equal to $\frac{ \pi}{\log  r^2+\pi}$. On the other hand, Theorem \ref{tctokh} guarantees that the Diederich-Forn{\ae}ss Index is at least $\frac{ \pi}{\log  r^2+\pi}$, so the result is sharp.
  			\end{proof}

\section{Existence of a Family of Good Vector Fields}

  In this section, we will relate the Diederich-Forn{\ae}ss Index to other sufficient conditions for global regularity.  In \cite{BoSt92}, Boas and Straube defined Hartogs domains that were ``nowhere worm-like," and showed that Condition $R$ is obtained on such domains.  In the notation of the present paper, $\Omega$ is nowhere worm-like if and only if $\frac{\partial^2\tilde\delta}{\partial z\partial\bar w}=0$ on any annulus in the boundary of $\Omega$.  Therefore, Theorem \ref{tctokh} implies that the Diederich-Forn{\ae}ss Index for a nowhere worm-like domain is equal to one.

	\begin{proof}[Proof of Theorem \ref{thm:good_vector_field}]
If a family of good vector fields exists, then the hypotheses of Theorem 2.11 in \cite{Har18} are satisfied, and hence the Diederich-Forn{\ae}ss Index of $\Omega$ is equal to one.
		
		Conversely, suppose that the Diederich-Forn{\ae}ss Index of $\Omega$ is equal to one.  We will consider each connected component of the set of weakly pseudoconvex points separately.  Since a disk is simply-connected, results of Boas and Straube \cite{BoSt93} can be used immediately to prove the existence of a family of good vector fields in a neighborhood of a disk in the boundary.  For sets $K$ satisfying the hypotheses of Lemma \ref{fono}, we have a defining function that is plurisubharmonic on $K$, so results of Boas and Straube \cite{BoSt91} can be used to prove the existence of a family of good vector fields.  This leaves us with the case of an annulus in the boundary.

	From Theorem \ref{main}, if the Diederich-Forn{\ae}ss Index is equal to one, then for any annulus $M$ in $b\Omega$, the constants $A$, $B$, and $C$ given by Theorem \ref{main} must either satisfy $A=B$ or for every $C>0$ there exists $(z_C,w_C)\in M$ such that $\abs{\frac{\partial^2\tilde\delta}{\partial w\partial\bar z}(z_C,w_C)}\leq\frac{C}{|z_C|}$.  When $A=B$, the result follows from Example 3 in \cite{BoSt92}.  In the other case, compactness of $M$ (and $z\neq 0$ on $M$) guarantees the existence of $(z_0,w_0)\in M$ such that $\abs{\frac{\bd^2\tilde{\delta}}{\bd w\bd\bar{z}}(z_0,w_0)}=0$.  Due to the circular symmetry of the Hartogs domain, $\frac{\bd^2\tilde{\delta}}{\bd w\bd\bar{z}}(z_0e^{i\theta},w_0)=0$ for all $\theta\in\mathbb{R}$. Hence D'Angelo's 1-form $\alpha$ (see \eqref{alfi}) satisfies $\alpha=0$ on some circle in $M$.  This shows that the cohomology class represented by $\alpha$ is trivial on $M$.  By Remark 5 in Section 4 of \cite{BoSt93}, there exists a family of good vector fields in a neighborhood of $M$.

  	  \end{proof}

  	 \section{A Necessary Condition for the Existence of a Strong Stein Neighborhood Basis }
  	
  	  In \cite{BeFo78}, Bedford and Forn{\ae}ss introduced a general criteria for the existence of a Stein neighborhood basis.  On a Hartogs domain, this criteria is relatively easy to compute explicitly.  Let $M$ be the annulus in $b\Omega$ given in the statement of Theorem \ref{ostnb}.  Let $\gamma_1$ denote the boundary component of $M$ parameterized by $(\sqrt{A}e^{i\theta},w)$ for $\theta\in\mathbb{R}$.  For $\alpha$ defined by \eqref{alfi}, we wish to compute $c_1=\int_{\gamma_1}\alpha$.  We have
  	  \[
  	  c_1=\int_{\gamma_1=\sqrt{ A}e^{i\theta}} -\frac{\bd^2\tilde{\delta}}{\bd z\bd\bar{w}}(z,w)
  	  /\frac{\bd\tilde{\delta}}{\bd\bar{w}}(z,w)dz-\frac{\bd^2\tilde{\delta}}{\bd \bar{ z}\bd w}(z,w)
  	  /\frac{\bd\tilde{\delta}}{\bd w}(z,w) d\bar{z}
  \]
  Using \eqref{eq:bar_w_z_derivative}, we have
  \begin{align*}
  	   c_1&=\int_{\gamma_1=\sqrt{ A} e^{i\theta}}-\frac{\rho_{t\bar{ w}}(\abs{ z}^2,w)\bar{z}}{\rho_{\bar{w}}(\abs{z}^2,w)}dz-\frac{\rho_{t w}(\abs{ z}^2,w)z}{\rho_{w}(\abs{z}^2,w)}d\bar z \\
  &=\int_0^{2\pi}-2\re\left(\frac{\rho_{t\bar{ w}}(A,w)}{\rho_{\bar{w}}(A,w)}i A \right)d\theta.
  \end{align*}
  Since the integrand is now constant with respect to $\theta$, we have
\begin{equation}\label{town}
  	   c_1=-2\re\left(2\pi i A\frac{\rho_{t\bar{ w}}(A,w)}{\rho_{\bar{w}}(A,w)}\right).
  	  \end{equation}

Let $\omega(z)=\frac{c_1}{4\pi}\log\frac{\abs{z}^2}{B}$, and let $\gamma_0$ denote the boundary component of $M$ parameterized by $(\sqrt{B}e^{i\theta},w)$ for $\theta\in\mathbb{R}$.  Clearly $\omega$ is harmonic and $\omega(z)=0$ for $z\in\gamma_0$.  If we define $d^c\omega=i(\bar{\bd}-\bd)\omega$, then
\[
  	  	\int_{\gamma_1} d^c\omega = i\int_{\gamma_1}\frac{c_1}{4\pi\bar z}d\bar{z}-i\int_{\gamma_1}\frac{ c_1}{4\pi z}dz=c_1.
  \]
When $\abs{z}^2=A$, we have $\omega(z)=a_1$, where $a_1$ is the constant given by
\begin{equation}
\label{eq:a_1}
  a_1=\frac{c_1}{4\pi}\log\frac{A}{B}.
\end{equation}
Using Bedford and Forn{\ae}ss's main result in \cite{BeFo78}, we immediately obtain
  	  \begin{thm}[Bedford and Forn{\ae}ss]\label{Bed}
  	  	Let $ \Omega\subset\mathbb{C}^2$ be a Hartogs domain with $C^4$ boundary that is strongly pseudoconvex except on an annulus $M=\{(z,w):A\leq|z|^2\leq B\}\subset b\Omega$.  Let $c_1$ and $a_1$ be given by \eqref{town} and \eqref{eq:a_1}.  If $\abs{a_1}<\pi$ then $\bar{\Omega}$ admits a Stein neighborhood basis, and if $\abs{a_1}>\pi$ then $\bar{\Omega}$ does not admit a Stein neighborhood basis.  	  	
  	  \end{thm}

  	
  We first carry out some computations to rephrase Bedford and Forn{\ae}ss's result in terms of the curvature term that we have been studying in this paper:
  	  \begin{proof}[Proof of Theorem \ref{ostnb}]

  	  First, we assume we have $\abs{\frac{ \bd^2\tilde{\delta}}{\bd z\bd\bar{w}}(z,w)}<\frac{ \pi}{2\sqrt{ A}\abs{\log\frac{ A}{B}}}$ when $|z|^2=A$. Using \eqref{eq:bar_w_z_derivative}, we have   	  	
  	  	$$\abs{\frac{ \bd^2\tilde{\delta}}{\bd z\bd\bar{w}}(z,w)}=\abs{ \rho_{t\bar{w}}(A,w)\bar{z}}=\abs{ \rho_{t\bar{w}}(A,w)}\sqrt{ A} . $$
  	  	Using \eqref{eq:deltaw} and \eqref{town}, we have $\abs{c_1}< 8\pi A\frac{ \abs{\frac{ \bd^2\tilde{\delta}}{\bd z\bd\bar{w}}(z,w)}}{\sqrt{A}}$, so \eqref{eq:a_1} gives us
  	  	\[
   |a_1|=\abs{ \frac{ c_1}{4\pi}\log\frac{ A}{B}}< 2 \sqrt{A}\abs{\frac{ \bd^2\tilde{\delta}}{\bd z\bd\bar{w}}(z,w)}\abs{ \log\frac{A}{B}} <\pi.
  \]
  	  	By Theorem \ref{Bed}, this implies a Stein neighborhood basis exists.

 Next, we assume  $\abs{\frac{\bd^2\til{\delta}}{\bd z\bd\bar{w}}(z,w)}>\frac{\pi}{2\sqrt{A}\abs{\log\frac{A}{B}}}$ when $\abs{z}^2=A$.
 Observe that the Taylor series in $w$ for $\til{\delta}$ near $M$ must be of the form
 $$ \til{\delta}(z,w)=Re (w e^{i\theta(\abs{z}^2)})+O(\abs{ w}^2),$$ for some smooth real-valued function $\theta.$
 On $M$, we have
 \[ \frac{\bd\til{\delta}}{\bd w}(z,w)=\frac{1}{2}e^{i\theta(\abs{z}^2)}\]
 and
 \[ \frac{\bd^2\til{\delta}}{\bd \bar{z}\bd w}(z,w)=\frac{1}{2}e^{i\theta(\abs{z}^2)} i \theta^\prime (\abs{z}^2)z.\]
 When $|z|^2=A$, we may substitute this into \eqref{town} to obtain
 \[
   c_1=-4\pi A\theta^\prime(A).
 \]
 On $M,$
 \[\abs{\frac{\bd^2\til{\delta}}{\bd\bar{z}\bd w}(z,w)}=\frac{1}{2}\abs{\theta^\prime(\abs{z}^2)}\abs{z}, \]
 so when $|z|^2=A$ we have
\[
  \abs{c_1}=8\pi\sqrt{A}\abs{\frac{\bd^2\til{\delta}}{\bd\bar{z}\bd w}(z,w)}.
\]
 Using \eqref{eq:a_1}, our hypothesis implies that $\abs{a_1}>\pi$, so no Stein neighborhood basis exists by Theorem \ref{Bed}.

  	  	\end{proof}	

  \begin{proof}[Proof of Corollary \ref{cor:st_nbhd_existence}]
    Let $M=\{(z,w):A\leq|z|^2\leq B\}$ be an annulus in the boundary of $\Omega$.  Since the Diederich-Forn{\ae}ss Index of $\Omega$ is equal to one, we may apply Theorem \ref{main} to the annulus $M_\eps=\{(z,w):A\leq|z|^2\leq A+\eps\}$ for any $\eps>0$ and conclude that for every $\eps>0$ there exists a circle $(r_\eps e^{i\theta},w)\subset M_\eps$ on which $\frac{\partial^2\tilde\delta}{\partial w\partial\bar z}=0$.  By continuity, $\frac{\partial^2\tilde\delta}{\partial w\partial\bar z}=0$ when $|z|^2=A$.  Hence, Theorem \ref{ostnb} guarantees the existence of a Stein neighborhood basis for $\Omega$.
  \end{proof}

  	  	\begin{proof}[Proof of Theorem \ref{tntn1}]
   We define $\hat h$ as in the proof of Theorem \ref{main}.  Following the proof of that Theorem, we see that the complex Hessian of $\sigma$ is positive semi-definite on $\Omega\cap U$ for a neighborhood $U$ of $p$ in $M$ only if the matrix $\begin{pmatrix}a(p)&c(p)\\\bar c(p)&b(p)\end{pmatrix}$ is positive semi-definite, where
   \begin{align*}
     a(p)&=-4\hat{h}(p)\tau\abs{ \rho_{tw}(p)}^2|p_z|^2+\frac{ \bd^2\hat{ h}}{\bd z\bd\overline{ z}}(p),\\
     b(p)&=\frac{ 1}{4} \tau(\tau-1)\hat h(p)\text{, and}\\
     c(p)&=\hat{h}(p)\tau\rho_{t\overline{ w}}(p)\bar{p}_z+\tau\frac{ \bd\hat{ h}}{\bd z}(p)\frac{ \bd\til{\delta}}{\bd\bar{w}}(p).
   \end{align*}
   Since $\tau>1$ implies $b(p)>0$, it suffices to check $a(p)b(p)-|c(p)|^2\geq 0$.  If we again write $p_t=|z|^2$ and $\hat h=g^{1-\tau}$ for some $g>0$, this is equivalent to
  	  		\begin{equation} \label{eq:tau}
  	  		-\frac{ 1}{4}\tau(1-\tau)^2 \left(\frac{ \bd g}{\bd t}(p)+p_t\frac{ \bd^2 g}{\bd t^2}(p)\right)-\tau^3 p_t g(p) \abs{ \rho_{tw}(p)}^2\geq 0.
  	  		\end{equation}
  	  		
  	  		Our assumptions imply $\abs{\rho_{tw}(p) }>\frac{C}{t}$, so \eqref{eq:tau} implies
  	  		\[-\frac{ 1}{4}\tau(1-\tau)^2 \left(\frac{ \bd g}{\bd t}+t\frac{ \bd^2 g}{\bd t^2}\right)-\tau^3 tg\left(\frac{C^2}{t^2}\right)g> 0,
  	  		\]
  for $A\leq t\leq B$, where we have again omitted the dependence on $p$ since this no longer depends on $\arg p_z$ or $p_w$.

  	  		We will show that this implies a contradiction. Let $\tilde{g}(s)=g(e^s)$, i.e., we will use the substitution $t=e^s$ where $\log A\leq s\leq\log B$ .  Then \eqref{eq:tau} takes the form
  	  		\begin{equation}\label{eq:off}
  	  		\frac{ 1}{4}\tau(1-\tau)^2 \left(\frac{ d^2\tilde{ g} }{d s^2}\right)+\tau^3 C^2\tilde{ g} < 0.
  	  		\end{equation}
  	  		After making the substitution $u=\sqrt{\frac{ \tau^3 C^2}{\frac{ 1}{4}\tau(1-\tau)^2}}s=\frac{2\tau C}{\tau-1}s$, where $$\frac{2\tau C}{(\tau-1)}\log A<u<\frac{2\tau C}{(\tau-1)}\log B,$$ \eqref{eq:off} implies
  	  		\begin{equation} \label{eq:seva}
  	  		\frac{ d^2\til{g}}{du^2}+\tilde{g}<0.
  	  		\end{equation}
  	  		If $ \frac{2\tau C}{(\tau-1)}\log\frac{B}{A}\geq \pi$, then we obtain a contradiction as in the proof of Theorem \ref{main}.

  	  		Therefore, we must have $\frac{\tau}{\tau-1}< \frac{ \pi}{2 C\log \frac{B}{A}}$.  Since $\tau>1 $, we must have
  \[
    1-\frac{2C\log\frac{B}{A}}{\pi}>0 \text{ and }\tau>\frac{1}{1-\frac{2C\log\frac{B}{A}}{\pi}}.
    \]
  	  		
  	  		\end{proof}
\begin{proof}[Proof of Corollary \ref{simsim}]
	
	As in the proof of Corollary \ref{cor:worm_domain}, we know that
\[
  \abs{\frac{\partial^2\tilde\delta}{\partial\bar z\partial w}(z,w)}=\frac{1}{2|z|}
  \]
  on $M$.  Since we have parameterized the worm domain so that $B=r^2$ and $A=1$, then Theorem \ref{ostnb} guarantees the existence of a Stein neighborhood basis if $\frac{1}{2}<\frac{\pi}{2\log r^2}$, and no Stein neighborhood basis can exist if $\frac{1}{2}>\frac{\pi}{2\log r^2}$.  Since the hypotheses of Theorem \ref{tntn1} are satisfied for any $C<\frac{1}{2}$, the existence of a strictly positive function $h$ such that $\sigma$ is plurisubharmonic on some neighborhood of $\bar\Omega$ will imply that $\frac{\pi}{2\log r^2}\geq\frac{1}{2}$ and $\tau\geq\frac{\pi}{\pi-\log r^2}$.
	\end{proof}
	

\bibliographystyle{amsplain}
\bibliography{harrington}
\end{document}